\documentclass[11pt,reqno]{amsart}
\usepackage{amssymb,amsfonts,amsthm,amscd,stmaryrd,dsfont,esint,upgreek,constants,todonotes}
\usepackage[mathcal]{euscript}
\usepackage[utf8]{inputenc} 
\usepackage{comment}  
\usepackage{fancyhdr} 
\usepackage[mathcal]{euscript}
\usepackage{graphicx,color}
\usepackage{lipsum}
\numberwithin{equation}{section}

\newcommand{\R}{\mathbb R}
\def\E{\mathbb E}

\newconstantfamily{C}{symbol=C}
\newconstantfamily{c}{symbol=c}
\newconstantfamily{O}{symbol=\Omega}

\def\XXint#1#2#3{{\setbox0=\hbox{$#1{#2#3}{\int}$}
\vcenter{\hbox{$#2#3$}}\kern-.5\wd0}}

\numberwithin{equation}{section}
\newtheorem{thm}{Theorem}[section]
\newtheorem{lem}[thm]{Lemma}

\newtheorem{prop}[thm]{Proposition}
\newtheorem{assumption}[thm]{Assumption}
\theoremstyle{definition}
\newtheorem{defn}[thm]{Definition}
\newtheorem{rmk}[thm]{Remark}

\def\smallnegint{\mathop{\int\mkern-13mu
        \raise.5ex\hbox{${\scriptscriptstyle\diagup}$}}\nolimits}

\def\bx{{\bf x}}
\def\by{{\bf y}}
\def\ssetminus{\,\raise.4ex\hbox{$\scriptstyle\setminus$}\,}

\newcommand{\be}{\begin{equation}}
\newcommand{\ee}{\end{equation}}

\def\dive{{\rm div}}

\renewcommand{\bar}{\overline}
\renewcommand{\tilde}{\widetilde}
\renewcommand{\hat}{\widehat}

\begin{document}
\title[Approximately optimal distributed controls]{Approximately optimal distributed controls for high-dimensional stochastic systems with pairwise interaction through controls}

\author{Elise Devey}
\thanks{INRIA Paris,
48 rue Barrault, CS 61534,
75647 Paris Cedex, Université Paris-Dauphine, elise.devey@inria.fr.\\ The author is grateful to Pierre Cardaliaguet for his valuable and insightful contributions to this work.}
\begin{abstract}
This paper investigates large-population stochastic control problems in which agents share their state information and cooperate to minimize a convex cost functional. The latter is decomposed into individual and coupling costs, with the distinctive feature that the coupling term is a pairwise interaction function between the controls. To address this setting, we follow closely \cite{jackson2023approximatelyoptimaldistributedstochastic}: we introduce a related problem where each agent observes only its own state. We then establish a quantitative bound on the difference between the value functions associated with these two problems. We obtain this result by reformulating the problems analytically as Hamilton–Jacobi type equations and comparing their associated Hamiltonians. The main difficulty of our approach lies in establishing a precise comparison between the distributions of the corresponding optimal controls.
\end{abstract}

\maketitle      


\section{Introduction}
Understanding large populations and complex systems is a central challenge in mathematical modeling, with many real-world applications, from social networks and power grid infrastructures to financial markets. In particular, many optimization problems can be described as high-dimensional stochastic optimal control problems in which many agents cooperate to minimize a cost functional.\\

In the context of our paper, we consider the following mathematical setup, directly inspired by the problems of managing the flexibility of the electrical grid. The main motivation is to extend existing high-dimensional stochastic control problems, such as coordinating flexible electrical appliances to support power system operation (see \cite{de2019mean,Seguret_2023}), by incorporating constraints from the electricity network.   For a time horizon $T >0\,,$ each of \( N \) agents \( i = 1, \ldots, N \) controls a state process \( X^i_t \), with values in \( \mathbb{R}^d \), governed by the stochastic dynamics
\begin{equation}\label{eq:1.1}
dX^i_t = \alpha^i(t, {\bf X_t}) \, dt + dW^i_t.
\end{equation}
Here, \( W^1, \ldots, W^N \) are independent Brownian motions, and \( \alpha = (\alpha^1, \ldots, \alpha^N) \) is a Markovian (feedback) control which depends on the entire vector \( X_t = (X^1_t, \ldots, X^N_t) \) of states. That is, \( \alpha \) belongs to the set of full-information controls, defined as follows.\\
\textbf{Full-information controls:} \(\mathcal{A}\) denotes the set of \(\alpha = (\alpha^1, \ldots, \alpha^N)\), where 
\[
\alpha^i : [0,T] \times (\mathbb{R}^d)^N \to \mathbb{R}^d
\]
is measurable for each \(i\), and the SDE \eqref{eq:1.1} admits a unique strong solution.\\
Therefore, the Full-Information Problem is a minimization of a convex functional and can be formulated as follows.
\[
\mathcal{P} : \inf_{\alpha \in \mathcal{A}} J_0(\alpha)\]
where, \[\quad
J_0(\alpha)=\E\biggl[\int_0^T \biggl(\frac{1}{2N}\sum_i \vert \alpha^i(t,\boldsymbol{X^{\alpha}_t} )\vert^2 
    +
    f^N(\boldsymbol{\alpha}(t,\boldsymbol{X^{\alpha}_t}))\biggr)dt+g^N(\boldsymbol{X^{\alpha}_T})\biggr]\,.\]

Notice that the presence of the interaction function $f^N$
 makes it impossible to solve each agent’s problem separately, as their behaviors are interdependent.\\
However, depending on the structure of $f^N$, one can argue that when the number of agents $N$ is large, finding the optimal control by treating each agent independently can provide a good approximation. The objective of this paper is to quantify the resulting loss in the cost function associated with the transition from centralized to decentralized optimization.\\

 \textbf{Introduction to Mean Field Games/Control and their Extensions}\\
 Analyzing Pareto optimality directly in N-player models is notoriously difficult because of their inherent complexity, and a similar difficulty arises in competitive control problems focused on Nash equilibria. To address these challenges, the theories of Mean Field Games (MFGs) and Mean Field Control (MFC) emerged about twenty years ago, targeting competitive and cooperative settings, respectively, see \cite{huang2006large, Lasry2007MeanFG}. Inspired by statistical physics research, these frameworks aim to describe the asymptotic behavior of Nash equilibria or Pareto optima in large populations of symmetric stochastic differential games as the number of agents tends to infinity. More precisely, these frameworks apply to exchangeable systems, that is, systems of symmetric agents whose interactions with others occur only through the empirical distribution of their state variables.
 In the asymptotic regime, the curse of dimensionality is alleviated, and the equilibria or optima can be characterized through the so-called MFG or MFC systems. The intuitive idea is that state processes should become approximately i.i.d. as $N \to \infty$, and therefore their
empirical measure $m^N_{{\bf X_t}}$ should be close to the common law $\mathcal{L}({\bf X_t})$ of state processes, by
a law of large numbers. Therefore, it leads to models in which an agent’s representative state depends on its own distribution, giving rise to the class of McKean–Vlasov (MKV) equations. \\
If $
\alpha^{MF} : [0,T] \times \mathbb{R}^d \to \mathbb{R}^d $
denotes an optimal control for this mean field control problem, then the controls 
$$(\alpha^{MF,i})_{i=1}^N \in \mathcal{A}, \quad \text{defined by } \alpha^{MF,i}(t, x) := \alpha^{MF}(t,x^i),$$
should be nearly optimal for the original control problem.
For an in-depth exposition, we point the reader to the foundational lectures of P.-L. Lions \cite{Lasry2007MeanFG} and the extensive monographs by Carmona and Delarue \cite{carmona2018probabilistic}. These works lay out the core mathematical techniques for tackling such problems, including Itô calculus along probability measure flows, the stochastic maximum principle, forward–backward SDEs of McKean–Vlasov type, and the Master Bellman equation in Wasserstein space.
\\

More general MFG systems, referred to as Extended MFG models or Mean Field Games of Controls, were introduced in \cite{gomes2014existence,gomes2015extendeddeterministicmeanfieldgames}. These models serve as an asymptotic framework for systems in which agents interact through both their state and control variables, as considered in our setup. For probabilistic formulations of these models, see \cite{carmona2015probabilistic,chassagneux2014probabilistic,djete2022mckean,carmona2018probabilistic}, and for analytic methods, we refer to \cite{kobeissi2022mean,kobeissi2022classical,achdou2020mean,santambrogio2021cucker,doi:10.1137/21M1407720,graber2021weak,graber2023master,camilli2023quasi,cardaliaguet2018mean}. The counterpart for MFC framework, known as Extended Mean Field Control (EMFC) has been studied in \cite{yong2013linear,graber2016linear,li2019linear,acciaio2019extended,pham2018bellman}. However, our model involves non-exchangeable agents arising from heterogeneous interactions and therefore cannot be approximated by this extended framework.\\

There have been some proposals
to extend the mean field framework to accommodate certain models with heterogeneity, more precisely when interactions depend on a network structure. Some authors have taken advantage of the theory of graphon to model heterogeneous interactions, \cite{caines2021graphon,bayraktar2023propagation, lacker2021soret,aurell2022stochastic,gao2020linear}. This framework is called Graphon Mean Field Games. It requires that the graph representing the interactions of agents possesses certain asymptotic structure, namely it has to be dense so that it converges to a graphon as the number of nodes tends to infinity. Another approach developped in \cite{bertucci2025strategicgeometricgraphsmean} consists in modeling  the interactions through a Riemannian geometric graph and deriving the asymptotic problem as a mean-field game with a Riemannian-based interaction structure. It allows to consider a dynamic structure of interaction that depends
on the interaction of the players themselves. \\

\textbf{The Convergence Problem in Mean Field Games and Control}\\
One of the main difficulties in the early development of the theory was to rigorously characterize and measure what qualifies as a “nearly optimal” solution in this setting.
The convergence problem in MFG and MFC addresses the precise relationship between 
$N$-player games and their mean field limits, with particular emphasis on how equilibria or optima, respectively, behave as the number of players grows. In cases where there exists a unique solution, one aims to demonstrate that the equilibrium/optimum of the finite 
$N$-player problem approach this limiting equilibrium/optimum in an appropriate sense. More generally, people have been interested in analyzing the convergence of the value functions $V^N$ and $V^{MF}$ respectively associated with finite and asymptotic models.\\ 

In the mean field control setting, if the running and terminal cost functions are convex and sufficiently smooth, the convergence problem has already been addressed quantitatively, with the optimal rate shown to be of order $N^{-1}\,,$ $\vert V^{N}-V^{MF}\vert =\mathrm{O}(N^{-1})\,,$ see \cite{germain2022rate}.
In the case of Mean Field Games, we replace the convexity by a monotonicity condition and apply the same strategy, see \cite{cardaliaguet2019master, carmona2018probabilistic}. Recently, Jackson and Mészáros provided a quantitative rate of convergence in the Mean Field Games of Control setting. Specifically, for a dimension $d >4\,,$ and an initial distribution ${\bf m_0} \in\mathcal{P}_2(\mathbb{R}^d)^N,$ they
established a rate of order $N^{\frac{-2}{d}}\,.$ \\

Answering similar questions in the absence of
structural conditions like convexity and monotonicity has been done qualitatively using compactness arguments, first for MFC by Lacker in \cite{lacker2016limittheorycontrolledmckeanvlasov}, then extended to Extended Mean Field Control by Djete in the series of works \cite{djete2022mckean, djete2022extendedmeanfieldcontrol, djete2023large}. More recently, \cite{cardaliaguet2023algebraic} followed by \cite{daudin2024optimal}, established quantitative algebraic convergence rates under 'natural' assumptions in the MFC setting. According to \cite{daudin2024optimal}, the fastest achievable convergence rate varies between $N^{\frac{-1}{d}}$ and $N^{\frac{-1}{2}}\,,$ depending on the metric used to measure the Lipschitz continuity of the value functions. For Graphon and Riemannian Mean Field Games, however, no convergence result has been proved yet.\\

\textbf{Our work and contribution}\\
In this paper, our aim is not to derive a suitable asymptotic problem and prove the convergence, but rather to demonstrate that the set of admissible controls may be restricted to a smaller subset without changing too much the minimal cost, which in particular includes the mean-field controls.
We define this subset as follows.\\

\textbf{Distributed controls:} \(\mathcal{A}_{\mathrm{dist}}\) denotes the set of \((\alpha^1, \ldots, \alpha^n) \in \mathcal{A}\) for which 
\[
\alpha^i(t, x_1, \ldots, x_n) = \alpha^i(t, x_i)
\]
depends only on the \(i\)-th state variable, for each \(i\).\\

Therefore, we alternatively define the distributed optimal control problem in order to compare it to the full-information problem. 
\[
 \mathcal{P}_{dist} : \inf_{\alpha \in \mathcal{A}_{\mathrm{dist}}} J_0(\alpha).
\]

While the full-information problem has been extensively investigated, see \cite{yong1999stochastic, fleming2006controlled, pham2009continuous}, the distributed control problem falls beyond the reach of the classical theory because of its nonstandard information constraints.\\

Still, if certain conditions ensure that the value functions associated with the two different problems
 remain close for large $N$, then one can restrict attention to distributed strategies, thereby reducing computational costs while still obtaining $\varepsilon$-optimal solutions.
In \cite{Seguret_2023}, the authors applied this idea to a high-dimensional convex stochastic control problem with interactions through the average of the agents’ states, showing that an 
$\varepsilon$-optimal distributed solution can be constructed without invoking the mean field limit. \\

Subsequently, Jackson and Lacker established a theoretical connection between the full-information and distributed problems in the case where agents interact through their states, see  \cite{jackson2023approximatelyoptimaldistributedstochastic}. In this context, they provide a sharp non-asymptotic bound on the gap between the value functions associated to the full-information and distributed problems. In other words, they quantify to what extend an optimal solution of the distributed control problem can be a good approximation for the optimal control of the full-information problem. \\

Our paper builds on their work and provides similar results in the context of interaction through the controls. 

The main difficulty of studying interaction through controls instead of interaction through states as in  \cite{jackson2023approximatelyoptimaldistributedstochastic} is that the optimal controls $(\alpha^1, \ldots, \alpha^N)$ have a less tractable structure, as we lose the decoupling property of the N Hamiltonians. So far, we have to restrict our study to pairwise and empirical average interactions through controls. \\
Thus, under convexity assumptions and for interaction functions of the form $$f^N : {\bf a} \to f_0(\frac{1}{N}\sum_ia^i)+\frac{1}{N^2}\sum_{i,j}\tilde h_{ij}(\vert a^i-a^j\vert)\,,$$ this paper establishes a quantitative bound between the two value functions $V^N$ and $V^N_{dist}$ respectively associated to the full-information and distributed control problem.
In the special case where $(g_{ij})_{i,j}$ are convex and $\max_{i,j}\Vert D_{ij} g^N \Vert_{\infty} \leq \frac{K}{N^{2}}$, for some positive constant $K\,,$
one obtains a convergence rate of order $N^{-1/2}$. Namely, we have
\begin{align*}
    0 \leq V^N_{\mathrm{dist}} - V^N &\leq \frac{M}{\sqrt{N}}\,,
\end{align*}
where $M$ depends on the partial and cross derivatives of $f^N$ and $g^N$ and is independent of $N\,.$
Note that these conditions on $g^N$ and $f^N$
 would correspond to the mean field regime in the case where the latter are symmetric.\\

\textbf{Outline of the paper.} The structure of the paper is as follows. We begin Section 2 by making precise assumptions, setting up the problem and presenting the main result. Section 3 gives some preliminary results about the optimal controls and the value function associated with the initial problem. Section 4 is finally dedicated to the proof of our main theorem.  
\section{Hypothesis and Main Theorem}

In this paper, all the processes are assumed to be defined on a complete filtered probability space 
$(\Omega, \mathcal{F}, \mathbb{F} = (\mathcal{F}_t)_{0 \leq t \leq T}, \mathbb{P})$, 
the filtration $\mathbb{F}$ satisfying the usual conditions, supporting N $d$-dimensional independent Wiener processes $\bigl((W^i_t)_{0 \leq t \leq T}\big)_{1\leq i\leq N}$. Let, for each $i \in \llbracket 1,N\rrbracket,$ $\mathbb{F}^i$ be the filtration generated by the process $(W^i_t)_{0 \leq t \leq T} $ and
    ${\bf Z}=(Z^1,...,Z^N)$ be a random vector such that  $Z^i$ is  $\mathbb{F}^i-$measurable  $\forall i \in \llbracket 1,N \rrbracket\,.$ \\

We will be working with the space $\mathcal{P}_2(\mathbb{R}^d)^N = \big(\mathcal{P}_2(\mathbb{R}^d)\big)^N,$ where $\mathcal{P}_2(\mathbb{R}^d)$ denotes the Wasserstein space of probability measures with finite second moment. 
We denote by $ m$ a generic element of $\mathcal{P}_2(\mathbb{R}^d)$ and by 
${\bf m} = (m_1,\dots,m_n)\,, $ a generic element of $\mathcal{P}_2(\mathbb{R}^d)^N$.\\
For $i = 1, \dots, N$, we denote by ${\bf m^{-i}}$ the element of 
$\mathcal{P}_2(\mathbb{R}^d)^{\,N-1}$ given by 
\[
{\bf m^{-i}} = (m_1, \dots, m_{i-1}, m_{i+1}, \dots, m_N).
\]

On both $\R$ and $\R^d,$ we write $\vert . \vert $ to denote the usual Euclidean norms  i.e. $\vert x^i \vert:=\sqrt{\sum_{k=1}^d\vert x^{i}_k\vert ^2} \,.$
On $(\R^d)^N\,,$ we will make use of three norms, we write :
\begin{itemize}
    \item $\Vert . \Vert$ to denote the Euclidean norm, i.e $\Vert {\bf x} \Vert:=\sqrt{\sum_{i=1}^N\sum_{k=1}^d\vert x^{i}_k\vert ^2} \,,$
    \item$ \Vert \cdot \Vert_{2,1}$ to denote the $L_{2,1}$ norm on $(\R^d)^N\,,$ i.e $$\Vert {\bf x} \Vert_{2,1}:=\sum_i\Vert x^i\Vert_2=\sum_{i=1}^N\sqrt{\sum_{k=1}^d\vert x^{i}_k\vert ^2} \,,$$ 
    \item $\Vert \cdot \Vert_{\infty}$ to denote the $\infty-$norm, i.e. $\Vert {\bf x} \Vert_{\infty}:= \max_{i \in \llbracket 1,N \rrbracket}\vert x^i\vert\,.$   
\end{itemize}
On $\R^{Nd\times Nd}\,,$ we write $\Vert \cdot \Vert_{op}$ to denote the operator norm, i.e. 
$
\Vert A \Vert_{op}:= \sup_{\|v\|=1} \|A v\|\,.
$\\
On $\mathcal{B}((\R^d)^N,(\R^d)^N)\,,$ we write $\Vert \cdot \Vert_{\infty}$ to denote the supremum norm on $(\R^d)^N$ of the $\infty-$norm on $(\R^d)^N$ , i.e.
\begin{align*}
    \Vert f \Vert_{\infty}:=\sup_{{\bf x} \in (\R^d)^N} \Vert f({\bf x}) \Vert_{\infty}=\sup_{{\bf x} \in (\R^d)^N}\max_{i \in \llbracket 1,N \rrbracket}\vert \big (f(x)\big )^i\vert\,.
\end{align*}

On $\mathcal{B}((\R^d)^N,\R^{d\times d})\,,$ we write $\Vert \cdot \Vert_{\infty}$ to denote the supremum norm on $(\R^d)^N$ of the Frobenius norm on $\R^{d\times d}$, i.e.
\begin{align*}
    \Vert f \Vert_{\infty}:=\sup_{{\bf x} \in (\R^d)^N} \sqrt{\sum_{k,l=1}^d\vert \big( f({\bf x})\big)_{k,l}\vert^2}\,.
\end{align*}

On $L^\infty(\R^d,m)\,,$ for some $ m \in \mathcal{P}_2(\R^d)\,,$ we write $\Vert \cdot \Vert_{L^\infty}$ to denote the $L^{\infty}-$norm with respect to the measure $ m\,,$ i.e. \[
\|f\|_{L^\infty} 
:= \inf\left\{ M \ge 0 \;:\; m\big(\{x \in \R^d : |f(x)| > M\}\big) = 0 \right\}.
\]

Let  $\mathcal{A}$ denotes the set of $\alpha = (\alpha^1, \dots, \alpha^N)$, where each $$\alpha^i: [0, T] \times (\mathbb{R}^d)^N \to \R^d\,,$$ is measurable and the SDE  
\begin{equation}
    \label{eq:state_strong}
    dX^{i,\alpha}_s= \alpha^i(s,\boldsymbol{X^{\alpha}_s})ds + dW^{i}_s \qquad s\in [t,T] \qquad X^{\alpha,i}_t=Z^i\,.
\end{equation}
admits a unique strong solution ${\bf X}$ satisfying
\begin{align*}
    \mathbb{E} &\left[ \int_t^T \left| \alpha_i(s, {\bf X_s}) \right|^2 ds \right] < \infty, \quad \text{for each } t \in [0, T] \,.
\end{align*}

So, given $\alpha \in \mathcal{A}$, we define the total cost functional $J_0(\alpha)$ associated with $\alpha$ by : 
\begin{align*}
    J_0(\alpha)=\E\biggl[\int_0^T \biggl(\frac{1}{2N}\sum_i \vert \alpha^i(t,\boldsymbol{X^{\alpha}_t} )\vert^2 
    +
    f^N(\boldsymbol{\alpha}(t,\boldsymbol{X^{\alpha}_t}))\biggr)dt+g^N(\boldsymbol{X^{\alpha}_T})\biggr]\,,
\end{align*}
where $\E$ denotes the expectation with respect to the given probability $\mathbb{P}$ and, for each i, $X^{\alpha,i}$ is solution to \eqref{eq:state_strong}.\\

Therefore, the goal is to minimize this cost functional over the set $\mathcal{A}$, namely to solve the following full-information problem
\begin{align}
    \label{eq:pb0}
    \inf_{\alpha \in \mathcal{A}}J_0(\alpha)\,. \tag{$P_{0}$}
\end{align}

We then define the associated value function 
\begin{align}
\label{def:valuefx0}
V^N:
    \begin{cases}
        [0,T]\times (\R^d)^N\to \R \\
        (s,\bx) \mapsto \inf_{\alpha \in \mathcal{A}}\E\biggl[\int_s^T \biggl(\frac{1}{2N}\sum_i \vert \alpha^i(t,\boldsymbol{X^{\alpha}_t} )\vert^2 
    +
    f^N(\boldsymbol{\alpha}(t,\boldsymbol{X^{\alpha}_t}))\biggr)dt+g^N(\boldsymbol{X^{\alpha}_T})\biggr]\,.
    \end{cases}  
\end{align}
where ${\bf X^{\alpha}}$ satisfies the SDE \eqref{eq:state_strong} and ${\bf X^{\alpha}_s}=\bx.$  \\

Let $\mathcal{A}_{dist} \subset \mathcal A$ denotes the set of controls \( (\alpha_1, \ldots, \alpha_n) \) for which, 
\[
\alpha_i(t, x_1, \ldots, x_n) = \alpha_i(t, x_i)
\]
depends only on the \( i \)-th state variable, for each \( i \).

An alternative problem to \eqref{eq:pb0} is to minimize the same cost functional but over the set $\mathcal{A}_{dist}$, namely solving 
\begin{align}
    \label{eq:pbdist}
    \inf_{\alpha \in \mathcal{A}_{dist}}J_0(\alpha)\,. \tag{$P_d$}
\end{align}

We also define the associated value function 
\begin{align}
\label{def:valuefxd}
\mathcal{V}_{dist}^N:
    \begin{cases}
        [0,T]\times \mathcal{P}_2(\mathbb{R}^d)^N\to \R \\
        (s,\boldsymbol \mu) \mapsto \inf_{\alpha \in \mathcal{A}_{dist}}\E\biggl[\int_s^T \biggl(\frac{1}{2N}\sum_i \vert \alpha^i(t,X^{\alpha,i}_t )\vert^2 
    +
    f^N(\boldsymbol{\alpha}(t,\boldsymbol{X^{\alpha}_t}))\biggr)dt+g^N(\boldsymbol{X^{\alpha}_T})\biggr]\,.
    \end{cases}
\end{align}
where ${\bf X^{\alpha}}$ satisfies the SDE \eqref{eq:state_strong} and ${\bf X^{\alpha}_s} \sim \boldsymbol \mu.$ 
\\

To be able to compare the full-information problem with the distributed problem, we introduce the lifted version of $V^N$,  ${\mathcal V}^N : [0,T] \times \mathcal{P}_2(\mathbb{R}^d)^N\to \R$ such that 
\begin{align}
    \label{def:lift_valuefx}
    {\mathcal V}^N(t, \boldsymbol \mu):= \int_{(\R^d)^N} V^N(t,\bx) \boldsymbol \mu(d\bx),\,.
\end{align}
Therefore, the idea is to 'lift' the full-information control problem from a state process in $(\R^d)^N$ to a state process in $\mathcal{P}_2(\mathbb{R}^d)^N\,,$  the space of vectors of probability measures in $\R^d\,.$

\begin{assumption}
    \label{hyp:fG}
    \begin{enumerate}
        \item The function $f^N : (\R^d)^N \to \R$ is of the form $$
        f^N({\bf a})=f_0(\frac{1}{N}\sum_ia^i)+\frac{1}{N^2}\sum_{i,j}h_{ij}( a^i-a^j )\,.$$  \\
        \item For any $(i,j) \in \llbracket 1,N\rrbracket^2,$ $h_{ij}$ is of the form 
        \begin{align*}
            h_{ij}:
            \begin{cases}
                \R^d \to \R_+\\
                a \mapsto \tilde h_{ij}(\vert a\vert)\,,
            \end{cases}
        \end{align*}
        where $\tilde h_{ij} : \R_+ \to \R_+$ is of class $C^2\,,$ convex, nondecreasing, such that $\tilde h_{ij}(0)=\tilde h'_{ij}(0)=0$ and $\Vert D^2h_{ij}\Vert_{\infty}$ is finite.\\
        \item The function $f_0 : \R^d \to \R$ is of class $C^2\,,$ convex, Lipschitz and $\Vert D^2f_0\Vert_\infty$ is finite. 
\item The function $g^N : (\R^d)^N \to \R$ is of class $C^2\,,$ bounded from below,  convex and has bounded derivatives of order two. We denote $C_G >0$ a constant such that the two inequalities hold : 
\begin{align*}
0 \leq  D^2 g^N  \leq \frac{C_G}{N} I_{Nd \times Nd}\,, \qquad\sup_i \Vert D_ {x_ i}g^N\Vert_{\infty}\leq \frac{C_G}{N}\,.
\end{align*}
    \end{enumerate}
\end{assumption}

\begin{rmk}
    The inequality $D^2 g^N  \leq \frac{C_G}{N} I_{Nd \times Nd}$ means that, for any $ {\bf x},{\bf y} \in (\R^d)^N, $ we have $$ {\bf y}^TD^2 g^N({\bf x}){\bf y}\leq \frac{C_G}{N}\sum_{i=1}^N\vert y^i\vert\,.$$
\end{rmk}

Before stating our main theorem, we define a concentration property that must be satisfied by the initial distribution. 

\begin{defn}[Poincaré inequality]
Let $m \in \mathcal{P}_2(\mathbb{R}^d)$ be a probability measure. One says that $m$ satisfies the Poincaré inequality with some constant $c$ if 
\[
\operatorname{Var}_m(g):=\int \vert g(x)-\int g(y)m(dy)\vert^2m(dx)  \leq c \int  |\nabla g(x)|^2 m(dx),
\]
for all bounded Lipschitz functions \( g : \mathbb{R}^d \to \mathbb{R} \). \\
\end{defn}
   
\begin{rmk}
    Here are some probability measures that satisfy the Poincaré inequality : 
    \begin{itemize}
        \item Any strongly log-concave probability measure (including Gaussian measures), that is, a measure of the form 
$$d\mu(x) = Z^{-1} e^{-V(x)}\,d\lambda(x)\,,$$ with $\nabla^2 V(x) \geq \lambda I_d$ for some $\lambda > 0$, 
satisfies a Poincaré inequality with constant $$c \le 1/\lambda\,;$$

        \item Any compactly supported probability measure whose support is smooth and convex satisfies a Poincaré inequality, with a constant depending only on the geometry of the support ;
        \item If,  $ \mu^1,....,\mu^N \in \mathcal{P}_2(\mathbb{R}^d)$ satisfy the Poincaré inequality with some constant $c^1,...,c^N$, then the product measure $$\boldsymbol \mu:=(\mu^1 \otimes \dots \otimes \mu^N) \in \mathcal{P}_2(\mathbb{R}^d)^N$$ satisfies the Poincaré inequality with constant $c:=\max_{i \in \llbracket 1, \dots N \rrbracket}c^i\,.$ 
    \end{itemize}
\end{rmk}

\begin{thm}
\label{thm:diff_V_O}
    Let $(t,\boldsymbol \mu) \in [0, T] \times \mathcal{P}_2(\mathbb{R}^d)^N$ such that $\boldsymbol \mu$ satisfies the Poincaré inequality with some non-negative constant $c_p$ and suppose that Assumption \ref{hyp:fG} holds.\\

    Then, if, for each $(i,j) \in \llbracket 1,N \rrbracket,$ 
    \begin{align}
        \label{hyp:gN}
    \Vert D_{ij}g^N \Vert_{\infty} \leq \frac{K_G}{N^2}\,, 
    \end{align}
    for some positive constant $K_G$ independent of N, we have 
    \begin{align*}
    0 \leq \mathcal{V}^N_{\mathrm{dist}}(t,\boldsymbol \mu) - \mathcal V^N(t,\boldsymbol \mu)
     & \leq \frac{M}{\sqrt{N}}\,,
\end{align*}
for some positive constant $M\,,$ independent of N. \\

Without the additional hypothesis \eqref{hyp:gN}, we have more generally
\begin{align*}
    0 \leq \mathcal{V}^N_{\mathrm{dist}}(t,\boldsymbol \mu) - \mathcal V^N(t,\boldsymbol \mu) \leq (T-t)\big (K_f(t)+K_g(t)\big )\,,
\end{align*}
where  
\begin{align*}
    \begin{cases}
        K_g(t):=e^{\frac{3C_G}{2N}(T-t)}\sqrt{NC_P\sum_{ij}\Vert D_{ij}g^N\Vert_{\infty}^2}\biggl(2C_G+e^{\frac{3C_G}{2N}(T-t)}\sqrt{NC_P\sum_{ij}\Vert D_{ij}g^N\Vert_{\infty}^2}\biggr)\\
    K_f(t):=  K_1+K_f'\frac{(C_G +\Vert Df_0\Vert_\infty)^2}{\sqrt{N}}\sqrt{(\frac{4}{N^2}\sum_{ij}\Vert D^2h_{ij}\Vert_{\infty}^2+\Vert D^2f_0\Vert_{\infty}^2)\big ( e^{3\frac{C_G}{N}(T-t)}-1
    \big )}\\
    K_1:= \frac{K_1'}{\sqrt{N}}(2\max_{ij}\Vert Dh_{ij}\Vert_{\infty,C_G+\Vert Df_0\Vert_{\infty}}+\Vert Df_0\Vert_{\infty})\\
    \qquad \qquad\times (C_G+\Vert Df_0\Vert_{\infty}) \big (\sqrt{\frac{4}{N^2}\sum_{ij}\Vert D^2h_{ij}\Vert_{\infty}^2}+\Vert D^2f_0\Vert_{\infty}\big )^2\,,
    \end{cases}
\end{align*}
with $C_G$ is defined in Assumption \ref{hyp:fG}, 
$$ C_p:=\frac{e^{2C_G T} - 1}{2C_G} + c_p e^{2C_G T}\,,$$
$K_1', K_f'\geq 0$  independent of N, and
\begin{align*}
    \max_{ij}\Vert Dh_{ij}\Vert_{\infty,C_G+\Vert Df_0\Vert_{\infty}}:=\max_{ij}\sup_{ x \in \R^d , \vert x \vert \leq C_G+\Vert Df_0\Vert_{\infty}}\vert Dh_{ij}(x)\vert \,.
\end{align*}
\end{thm}

\section{Preliminary Results }
In order to prove Theorem \ref{thm:diff_V_O}, we first state some useful results about the value function $V^N$.

\begin{lem}
    \label{lem:D2V}
    For each $0 \leq t \leq T$ and ${\bf x} \in (\mathbb{R}^d)^N$, the function $V^N(t, \cdot)$ is twice differentiable with  
\[
0 \leq D^2 V^N (t, {\bf x}) \leq \frac{C_G}{N} I_{Nd \times Nd},
\]
where $C_G$ is defined in item 4. of Assumption \ref{hyp:fG}.
\end{lem}

\begin{proof}
    First notice that the control problem under study is equivalent when posed over open-loop controls.\\
    We will first show that $V^N$ is convex. We define the function J as 
    \begin{align*}
        J(t,{\bf{x}},{\bf{a}}):=\E\biggl[\int_t^T\biggl(\sum_i\frac{1}{2N}(\vert a^{i}_s\vert^2+f^N({\bf{a}}_s)\biggr)ds+g^N({\bf{X^{a}_T}})\biggr]\,,
    \end{align*}
    where ${\bf a} = (a^1, \ldots, a^n) $ is a vector of square-integrable, adapted processes such that each $a^i$  is adapted to the  filtration $\mathbb{F}$  supporting N $d$-dimensional independent Wiener processes $\bigl((W^i_t)_{0 \leq t \leq T}\big)_{1\leq i\leq N}$, and where ${\bf{X^{a}}}$ satisfies the following SDE 
    \begin{align*}
        dX^{i,a}_t= a^i_tdt + dW^{i}_t\,,
    \end{align*} with ${\bf X_t^a}={\bf x}\,.$\\
    
    We first need to prove that $J(t,.,.)$ is convex on $(\R^d)^N \times \mathcal{A}$.\\
    Let ${\bf x}^0, {\bf x}^1 \in (\mathbb{R}^d)^N$ and ${\bf a}_0(\cdot), {\bf a}_1(\cdot) \in \mathcal{A}$.  \\
For $0 \leq \lambda \leq 1$, let  
$${\bf x}^\lambda = (1 - \lambda){\bf x}^0 + \lambda {\bf x}^1 \,,$$ 
$${\bf a}_\lambda(\cdot) = (1 - \lambda){\bf a}_0(\cdot) + \lambda {\bf a}_1(\cdot)\,.$$  
Let ${\bf x}_0(\cdot),{\bf x}_1(\cdot)$ and ${\bf x}_\lambda(\cdot)$ be the solutions to 
\begin{align*}
    d{\bf x}_\lambda(s)={\bf a}_\lambda(s)ds+dW(s)\,.
\end{align*}
for $t \leq s \leq t_1$, with ${\bf x}_\lambda(t) = {\bf x}^\lambda$. \\ 
By linearity of the drift with repect to the control, we have
${\bf x}_\lambda(s) = \lambda {\bf x}_1(s) + (1 - \lambda){\bf x}_0(s)$.\\
Moreover, $f^N$ is convex as a sum of convex functions according to item 2. and 3. of Assumption \ref{hyp:fG}.
    Together with the convexity of the quadratic function, we get the convexity of J. In other words, 
    \begin{align*}
        J(t,x^{\lambda},a_{\lambda})\leq(1-\lambda)J(t,{\bf x}^0,{\bf a}_0)+\lambda J(t,{\bf x}^1,{\bf a}_1)
    \end{align*}
    We now want to prove that $V^N(t,.)$ is convex on $(\R^d)^N$.\\
    Given $\delta >0$, choose ${\bf a}_0(.), {\bf a}_1(.)$ such that for $i=0,1$,
    \begin{align*}
        J(t,{\bf x}^i,{\bf a}_i)<V^N(t,{\bf x}^i)+\delta\,.
    \end{align*}
    Therefore,
    \begin{align*}
        V^N(t,{\bf x}^{\lambda})\leq J(t,{\bf x}^{\lambda},{\bf a}_{\lambda})\leq (1-\lambda)V^N(t,{\bf x}^0)+\lambda V^N(t,{\bf x}^1) + \delta\,.
    \end{align*}
    Since $\delta$ is arbitrary, $V(t,.)$ is convex.\\
    
    Secondly, we now prove the upper bound.\\
    We will use the fact that a $C^2$ function $g$ on a Euclidean space satisfies  
$D^2 g({\bf{x}}) \leq C I_{nd \times nd}$ for all ${\bf{x}}$, if and only if  
$$g(r {\bf{y}} + (1 - r) {\bf{z}}) \geq r g({\bf{y}}) + (1 - r) g({\bf{z}}) - \frac{C}{2} r(1 - r) |{\bf{y}} - {\bf{z}}|^2, \quad \forall {\bf{y}}, {\bf{z}}, \forall r \in (0,1)\,.$$\\
Fix $t \in [0, T]$ as well as ${\bf{x}}, {\bf{y}}, {\bf{z}} \in (\mathbb{R}^d)^N$ such that ${\bf{x}} = r {\bf{y}} + (1 - r) {\bf{z}}$ for some $r \in (0,1)$.  
Let ${\bf a^*}$ denote an open-loop optimizer of the control problem. Define $X$ by 
$$X_s^i=x^i+\int_t^sa^{*,i}_udu+(W_s^i - W_t^i)\quad t \leq s \leq T \,,$$ and ${\bf{Y}},{\bf{Z}} $ by  
\[
Y_s^i = y^i + \int_t^s a^{*,i}_u \, du + (W_s^i - W_t^i), \quad t \leq s \leq T,
\]
\[
Z_s^i = z^i + \int_t^s a^{*,i}_u \, du + (W_s^i - W_t^i), \quad t \leq s \leq T.
\]
So ${\bf{X}}=r{\bf{Y}}+(1-r){\bf{Z}}$ and ${\bf{Y}}-{\bf{Z}}={\bf{y}}-{\bf{z}}$\,. Moreover, 
\begin{align*}
    V^N(t,{\bf{x}})&=\E\int_t^T\biggl(\sum_i\frac{1}{2N}(a^{*,i}_s)^2+f^N(\boldsymbol{a^{*}_s})\biggr)ds+g^N({\bf{X_T}})\\
    &\geq \E\int_t^T\biggl(\sum_i\frac{1}{2N}(a^{*,i}_s)^2+f^N(\boldsymbol{a^{*}_s})\biggr)ds+rg^N({\bf{Y_T}})+(1-r)g^N({\bf{Z_T}})\\
    & \qquad \qquad -\frac{C_G}{2N}r(1-r)\vert {\bf{y}}-{\bf{z}}\vert ^2\\
    &\geq rV^N(t,{\bf{y}})+(1-r)V^N(t,{\bf{z}})-\frac{C_G}{2N}\vert {\bf{y}}-{\bf{z}}\vert ^2\,.
\end{align*}
and thus claimed estimate holds.

\end{proof}

\begin{lem}
    \label{lem:Lip_V}
    For each $0 \leq t_0 \leq T$, and $({\bf x_0},{\bf y_0})\in (\R^d)^N,$
    \begin{align*}
        \vert V^N(t_0,{\bf y_0})-V^N(t_0,{\bf x_0})\vert &\leq \frac{C_G}{N}\Vert{\bf y_0}-{\bf x_0} \Vert_{2,1}\,, 
    \end{align*}
        where the constant $C_G$ is defined in item 4. of Assumption \ref{hyp:fG}, in particular such that \begin{align*}
    \sup_i \Vert D_ {x_ i}g^N\Vert_{\infty}\leq \frac{C_G}{N}\,,
\end{align*}
and where we recall that $\Vert \cdot \Vert_{2,1}$ is defined by $$\Vert {\bf x} \Vert_{2,1}:=\sum_i\Vert x^i\Vert_2=\sum_{i=1}^N\sqrt{\sum_{k=1}^d\vert x^{i}_k\vert ^2} \,.$$
\end{lem}
\begin{proof}
    First notice that the control problem under study is equivalent when posed over open-loop controls.\\
 We recall the definition of the cost function J as 
    \begin{align*}
        J(t,{\bf{x}},{\bf{a}})=\E\biggl[\int_t^T\biggl(\sum_i\frac{1}{2N}(\vert a^{i}_s\vert^2+f^N({\bf{a}}_s)\biggr)ds+g^N({\bf{X^{a}_T}})\biggr]\,,
    \end{align*}
 where ${\bf a} = (a^1, \ldots, a^n) $ is a vector of square-integrable, adapted processes such that each $a^i$  is adapted to the  filtration $\mathbb{F}$  supporting N $d$-dimensional independent Wiener processes $\bigl((W^i_t)_{0 \leq t \leq T}\big)_{1\leq i\leq N}$, and where ${\bf{X^{a}}}$ satisfies the following SDE 
    \begin{align*}
        dX^{i,a}_t= a^i_tdt + dW^{i}_t\,,
    \end{align*} with ${\bf X_t^a}={\bf x}\,.$\\
    
    Let ${\bf x}_0, {\bf y}_0 \in (\mathbb{R}^d)^N$. Then, for all ${\bf a}(\cdot) \in \mathcal{A},$
\begin{align*}
    V^N(t_0,{\bf y}_0)\leq J(t_0,{\bf y_0},{\bf a})=J(t_0,{\bf x}_0,{\bf a})+ \E\biggl[g^N({\bf X_T^{a,y_0}})-g^N({\bf X_T^{a,x_0}})\biggr]\,.
\end{align*}
In particular for ${\bf a}={\bf a^*}$ denoting an open-loop optimizer of the control problem,
\begin{align*}
    V^N(t_0,{\bf y_0})\leq V^N(t_0,{\bf x_0})+ \E\biggl[g^N({\bf X_T^{a^*,y_0}})-g^N({\bf X_T^{a^*,x_0}})\biggr]\,.
\end{align*}

\begin{align*}
     V^N(t_0,{\bf y_0})-V^N(t_0,{\bf x_0}) &\leq \E(g^N({\bf X_T^{a^*,y_0}})-g^N({\bf X_T^{a^*,x_0}}))\\
    & \leq \E \big (\sum_i\Vert D_ig^N\Vert_{\infty}\vert X_T^{a^*,y_0,i}-X_T^{a^*,x_0,i}\vert\big )\\
    & \leq \frac{C_G}{N}\E \Vert {\bf X_T^{a^*,y_0}}-{\bf X_T^{a^*,x_0}} \Vert_{2,1}\\
    &\leq \frac{C_G}{N} \Vert {\bf y_0}-{\bf x_0} \Vert_{2,1}\,.
\end{align*}
where the third inequality comes from item 4. of Assumption \ref{hyp:fG} and the fourth because for each $i \in \llbracket 1,N \rrbracket,$
$$X^{a^*,x_0,i}_T=x_0^i+\int_{t_0}^Ta^{*,i}_udu+W_T^i-W_{t_0}^i\,,$$

$$X^{a^*,y_0,i}_T=y_0^i+\int_{t_0}^Ta^{*,i}_udu+W_T^i-W_{t_0}^i \,.$$

Therefore, $\vert V^N(t_0,{\bf y_0})-V^N(t_0,{\bf x_0})\vert \leq \frac{C_G}{N} \Vert {\bf y_0}-{\bf x_0} \Vert_{2,1}\,.$
\end{proof}

We then introduce two functions ${\bf \hat a}$ and ${\bf \check a} $ and provide some information and results.
\begin{defn}
    \label{def:hat_a_pt}
    For each ${\bf p} \in (\R^d)^N,$ define
\begin{align*}
    \hat{a}({\bf p}):=\text{argmax}_{{\bf a}\in (\mathbb{R}^d)^N}-\sum_{i=1}^N (p^i a^i +\frac{1}{2N} |a^i|^2) -f^N({\bf a})\,.
\end{align*}
In other words, $\hat{a}({\bf p})$ is the optimal solution to the problem 
$$H^N((p^i)_{i=1,\dots, N}):= \sup_{{\bf a}\in (\R^d)^N} -\sum_{i=1}^N (p^i a^i +\frac{1}{2N} |a^i|^2) -f^N({\bf a})\,.$$
\end{defn}

\begin{rmk}
\label{rmk:hat_a_fx}
One can notice that, for ${\bf p}=(p^i=p^i(\bx))_{i=1,\dots, N} \in (L^\infty((\R^d)^N))^N$,the function $\hat{a} \circ p$ is an optimal solution for \begin{align*}
\int_{(\R^d)^N} H^N&((p^i(x))_{i=1, \dots, N}){\bf m}(d\bx)\\ &= \sup_{{\bf a} \in \bigl(L^2_{{\bf m}}\big (\mathbb{R}^{dN}\bigr)\bigr)^N}\int_{(\R^d)^N} -\big ((\sum_i a^i(\bx) p^i(\bx) 
 +\frac{1}{2N} |a^i(\bx)|^2) + f^N({\bf a(x)})\big ){\bf m}(d\bx)\,,\end{align*}
where $\hat{a}$ is given in Definition \ref{def:hat_a_pt}.
\end{rmk}

\begin{defn}
\label{def:check_alpha}
    For $(q^i=q^i(x^i))_{i=1,\dots, N} \in (L^\infty(\R^d))^N$ and ${\bf m} \in (\mathcal P_2(\R^d))^N$, define 
    \begin{align*}
        \check  a({\bf q},{\bf m}) :=\text{argmax}_{{\bf a} \in \Pi_{i=1}^NL^2_{m^i}(\mathbb{R}^d)}-l({\bf a}, {\bf q}, {\bf m})\,,
    \end{align*}
    where
    $$
    l({\bf a}, {\bf q}, {\bf m}):=\int_{(\R^d)^N} (\sum_i (a^i(x^i) q^i(x^i)+\frac{1}{2N} |a^i(x^i)|^2) + f^N({\bf a(x)})) {\bf m}(d{\bf x})\,.$$
    In other words, $\check  a({\bf q},{\bf m})$ is the optimal solution to the problem 
    \begin{align*}
    \mathcal H^N&((q^i)_{i=1,\dots, N}, {\bf m}) \\&:= \sup_{{\bf a} \in \Pi_{i=1}^NL^2_{m^i}(\mathbb{R}^d)} -\int_{(\R^d)^N} (\sum_i (a^i(x^i) q^i(x^i)+\frac{1}{2N} |a^i(x^i)|^2) + f^N({\bf a(x)})) {\bf m}(d{\bf x})\,. 
\end{align*}
\end{defn}

\begin{rmk}
    $l$ is coercive and strictly convex with respect to ${\bf a}$ in $\Pi_{i=1}^NL^2_{m^i}(\mathbb{R}^d)$ , thus it admits a unique minimizer in $\Pi_{i=1}^NL^2_{m^i}(\mathbb{R}^d)$. Therefore, for any ${\bf q}\in (L^\infty(\R^d))^N$ and ${\bf m} \in (\mathcal P_2(\R^d))^N$, $\check  a({\bf q},{\bf m})$ is well defined in $\Pi_{i=1}^NL^2_{m^i}(\mathbb{R}^d)$. 
\end{rmk}
We will first give the necessary conditions of optimality satisfied by ${\bf \hat a}$ and ${\bf \check a} $. \\

Let ${\bf p} \in (\R^d)^N$ and, recalling Definition \ref{def:hat_a_pt}, let  ${\bf \hat a({\bf p})}$ be the optimal solution of the problem $H^N((p^i)_{i=1,\dots, N})\,.$ \\
Consequently, for each $i \in \llbracket 1,N\rrbracket$, $\hat a^i({\bf p})$ satisfies the fixed point relation: 
\begin{align}
\label{defhatapt}
\hat a^i({\bf p}) = -N p^i - N \partial_if^N \big ((\hat a^j({\bf p}))_{j=1, \dots N}\big ).
\end{align}

Similarly, ${\bf \check a}:={\bf \check a(p, {\bf m})}$, given in Definition \ref{def:check_alpha}, is an optimal solution of the problem $\mathcal H^N((p^i)_{i=1,\dots, N}, {\bf m})$. Then, for each $i \in \llbracket 1,N\rrbracket$, $\check a^i$ is characterized by the fixed point relation: 

\be\label{defchecka}
\check a^i(x^i) = -N p^i(x^i) - N  \int_{(\R^d)^{N-1}} \partial_if^N (\check a^i(x^i),\big(\check a^j(y^j)\big )_{j\neq i}){\bf m}^{-i}(d\by^{-i}),
\ee
for $x^i \in \operatorname{supp}(m^i)\,.$ By convention, we extend $\check a^i$ so that it satisfies the relation \eqref{defchecka} even outside the support of $m^i\,.$
\\

We next state some Lipschitz continuity for ${\bf \hat a}$ and ${\bf \check a} $.
\begin{lem}
    \label{lem:bd_diff_a}
    Recalling Definition \ref{def:hat_a_pt}, let, for ${\bf p} \in (\R^d)^N,$ ${\bf \hat a(p)}$ be the optimal solution to the problem 
    \begin{align*}
H^N((p^i)_{i=1,\dots, N}) = \sup_{{\bf a}\in (\R^d)^N} -\sum_{i=1}^N (p^i a^i +\frac{1}{2N} |a^i|^2) -f^N({\bf a})\,.
\end{align*}
    Then, for all ${\bf p, \tilde p} \in (\R^d)^N,$
    $$
    \sum_i\vert\hat a^i({\bf{p}}) - \hat a^i({\bf{\tilde p}}) \vert^2\leq N^2\sum_i\vert p^i - \tilde p^i\vert ^2\,.$$
\end{lem}
\begin{proof}
 As mentioned above, ${\bf \hat a(p)}$ satisfies the fixed-point equation \ref{defhatapt}.
Thus,    \begin{align*}
    \sum_i\vert\hat a^i({\bf{p}}) - \hat a^i({\bf{\tilde p}}) \vert^2&=-N\sum_i\big (\hat a^i({\bf{p}}) - \hat a^i({\bf{\tilde p}})\big)\big ( p^i - \tilde p^i+\partial_if(\hat a({\bf{p}}))-\partial_if(\hat a({\bf{\tilde p}})) \big)\\
    &=N\sum_i\big (\hat a^i({\bf{p}}) - \hat a^i({\bf{\tilde p}})\big)\big ( \tilde p^i - p^i\big)\\
    &-N\sum_i\big (\hat a^i({\bf{p}}) - \hat a^i({\bf{\tilde p}})\big)\big ( \partial_if( {\bf{\hat a(p)}})-\partial_if({\bf{\hat a(\tilde p)}}) \big)\\
    &\leq N\sum_i\big (\hat a^i({\bf{p}}) - \hat a^i({\bf{\tilde p}})\big)\big ( \tilde p^i - p^i\big)\\
    &\leq N\sqrt{\sum_i\vert\hat a^i({\bf{p}}) - \hat a^i({\bf{\tilde p}})  \vert^2}\sqrt{\sum_i\vert\tilde p^i - p^i\vert^2}\,,
\end{align*}
where the first inequality comes from the convexity of $f^N$ according to item 2. and 3. of Assumption \ref{hyp:fG}.
\end{proof}
\begin{lem}
    \label{lem:bd_diff_check_a}
    Let, for $(p^i=p^i(x^i))_{i=1,\dots, N} \in (L^\infty(\mathbb{R}^d))^N$ and ${\bf m} \in (\mathcal P_2(\mathbb{R}^d))^N,$
    ${\bf \check a(p,m)}$ be the optimal solution to the problem 
$$
\mathcal H^N((p^i)_{i=1,\dots, N}, {\bf m}) = \sup_{(a^i=a^i(x^i))} -\int_{(\R^d)^N} (\sum_i (a^i(x^i) p^i(x^i)+\frac{1}{2N} |a^i(x^i)|^2) + f^N({\bf a(\bx)})) {\bf m}(d{\bf x}). 
$$
    Then, for all ${\bf x, \tilde x} \in (\mathbb{R}^d)^N,$ and $i \in \llbracket 1,N \rrbracket$
    $$
    \vert \check a^i({\bf p, m})(x^i) - \check a^i({\bf{ p,m}})(\tilde x^i) \vert \leq N\vert p^i(x^i) -   p^i(\tilde x^i)\vert\,.$$
\end{lem}

\begin{proof}
The definition of ${\bf \check a}:={\bf \check a(p,m)}$ implies that it satisfies the fixed-point equation, for all ${\bf x} \in (\mathbb{R}^d)^N$: $$
\check a^i(x^i) = -N p^i(x^i) - N  \int_{(\mathbb{R}^d)^{N-1}} \partial_if^N ((\check a^j(x^i, {\bf y}^{-i}))_{j=1, \dots N}){\bf m}^{-i}(d{\bf y}^{-i}).
$$ 
Let ${\bf x}$,$ {\bf \tilde x} \in (\mathbb{R}^d)^N\,,$
    \begin{align*}
    &\vert \check a^i(x^i) - \check a^i(\tilde x^i) \vert^2\\
    &=-N\big (\check a^i(x^i) - \check a^i(\tilde x^i)\big)\cdot\big ( p^i(x^i) -  p^i(\tilde x^i)\big )\\
    & \quad -N\big (\check a^i(x^i) - \check  a^i(\tilde x^i)\big)\\
    & \qquad  \cdot \int_{(\mathbb{R}^d)^{N-1}} \big (\partial_if^N (\check a^i( x^i),(\check a^j(y^j))_{j \neq i})-\partial_if^N (\check a^i(\tilde x^i),(\check a^j(y^j))_{j \neq i}){\bf m}^{-i}(d{\bf y}^{-i})\\
    &\leq N\big (\check a^i(x^i) - \check a^i(\tilde x^i)\big)\cdot \big (p^i(\tilde x^i) - p^i( x^i)\big )\\
    &\leq N\vert \check a^i(x^i) - \check a^i(\tilde x^i)\vert  \vert  p^i(\tilde x^i) - p^i( x^i)\vert \,,
\end{align*}
where the first inequality comes from the convexity of $f^N$ with respect to each of its variables.
\end{proof}

We will now prove that, for each i, $\check a({\bf p}, {\bf m})^i$ is bounded under some condition on ${\bf p}\in (L^\infty(\R^d))^N$.
\begin{lem}
    \label{lem:check_a_bdd}
Let ${\bf m} \in \mathcal{P}_2(\mathbb{R}^d)^N,$ and  $\boldsymbol{p}=(p^i)_i \in (L^\infty(\mathbb{R}^d))^N$ satisfying 
\[
      \|p^i\|_{\infty}\leq \frac{C_G}{N} \; \forall i\,. 
\]
Let
$\check a({\bf p}, {\bf m})$ be introduced in Definition  \ref{def:check_alpha}.\\

Then, $\forall i \in \llbracket 1,N \rrbracket$, $\Vert \check a^i({\bf p},{\bf m}) \Vert_{L^\infty}\leq C_G+\Vert Df_0\Vert_{\infty}\,.$
\end{lem}

\begin{proof}

Fix ${\bf m} \in \mathcal{P}_2(\mathbb{R}^d)^N$ and ${\bf p} \in (L^{\infty}(\mathbb{R}^d))^N $.\\

Given ${\bf a}=(a^i)_i \in (L^2(\mathbb{R}^d))^N$, recall the definition of function $l$ given in Definition \ref{def:check_alpha}
\begin{align*}
l({\bf a}&,{\bf p},{\bf m}) =  \int_{(\mathbb{R}^d)^N} \biggl( \sum_i \left( a^i(x^i) p^i(x^i) + \frac{1}{2N} |a^i(x^i)|^2 \right) + f^N\big ({\bf a(x)}\big) \biggr) {\bf m(d x}) \\
&= \int_{(\mathbb{R}^d)^N} \biggl( \sum_i \left( a^i(x^i) p^i(x^i) + \frac{1}{2N} |a^i(x^i)|^2 \right) + f_0(\frac{1}{N}\sum_ia^i(x^i))\\
&\qquad \qquad+\frac{1}{N^2} \sum_{i,j} h_{ij}( a^i(x^i) - a^j(x^j) ) \biggr) {\bf m(d x})\\
&= \int_{(\mathbb{R}^d)^N} \biggl( \sum_i \left( a^i(x^i) p^i(x^i) + \frac{1}{2N} |a^i(x^i)|^2 \right) + f_0(\frac{1}{N}\sum_ia^i(x^i))\\
&\qquad \qquad+\frac{1}{N^2} \sum_{i,j} \tilde h_{ij}(\vert a^i(x^i) - a^j(x^j)\vert ) \biggr) {\bf m(d x})\,.
\end{align*}

Let ${\bf \check a(p,m)} \in (L^2(\mathbb{R}^d))^N$ be a maximizer of $-l(\cdot,{\bf p},{\bf m})\,,$ as defined in Definition \ref{def:check_alpha}, and denote it by ${\bf \check a}$ for simplicity.\\
Define $M:=C_G+\Vert Df_0\Vert_{\infty}\,$
and ${\bf \tilde a} \in (L^2(\mathbb{R}^d))^N$ such that, for any $i \in \llbracket 1,N \rrbracket$ and $ x^i \in \mathbb{R}^d$, 

\begin{align*}
\tilde{a}^i(x^i) := 
\begin{cases}
\check a^i(x^i) & \text{if } \vert a^i(x^i)\vert \leq M\,, \\
M \frac{\check a^i(x^i)}{\vert \check a^i(x^i)\vert }& \text{otherwise.} 
\end{cases}
\end{align*}
Notice that, for any $i \in \llbracket 1,N\rrbracket\,,$ and $x^i \in \R^d\,,$  $\tilde a^i(x^i)$ is the projection of $\check a^i(x^i)$ onto the closed convex set $B(0, M)$. 
Because any projection is $1-$Lipschitz, we get that for all $(i,j) \in \llbracket 1,N\rrbracket^2\,, $ and $(x^i,x^j) \in (\R^d)^2\,,$
$$
|\tilde a^i(x^i)-\tilde a^j(x^j)|\leq |\check a^i(x^i)-\check a^j(x^j)|.
$$
 According to Item 2. of Assumption \ref{hyp:fG}, the functions $(\tilde h_{ij})_{(i,j) \in \llbracket 1,N \rrbracket^2}$ are nondecreasing, so
\begin{align}
\label{ineq:hij}
 \tilde h_{ij}(|\tilde a^i(x^i)-\tilde a^j(x^j)|)\leq  \tilde h_{ij}(|\check a^i(x^i)-\check a^j(x^j)|). 
\end{align}
Moreover, using, for any $i \in \llbracket 1,N\rrbracket\,,$ and $x^i \in \R^d$, that $a^i \to p^i(x^i)\cdot a^i+\frac{1}{N}\vert a^i\vert^2 $ is convex and $f_0$ is Lipschitz by Item 3. of Assumption \ref{hyp:fG}, we have

\begin{align*}
& \sum_i (p^i(x^i)\cdot \check a^i(x^i) +\frac{1}{2N} |\check a^i(x^i)|^2) + f_0(\frac{1}{N} \sum_i \check  a^i(x^i))  \\
& \geq \sum_i (p^i(x^i)\cdot \tilde a^i(x^i) +\frac{1}{2N} |\tilde a^i(x^i)|^2) + f_0(\frac{1}{N} \sum_i \tilde  a^i(x^i)) +\sum_i p^i(x^i)\cdot (\check a^i(x^i)-\tilde a^i(x^i))\\
& \qquad  +\sum_i\frac{1}{N} \tilde a^i(x^i)\cdot(\check a^i(x^i)-\tilde a^i(x^i)) -\|Df_0\|_\infty\frac{1}{N} \sum_i |\check a^i(x^i)-\tilde a^i(x^i)| . 
\end{align*}
By assumption,  $\Vert p^i\Vert_{\infty}\leq \frac{C_G}{N}\,,$ and noticing that   $\tilde a^i(x^i)\cdot (\check a^i(x^i)-\tilde a^i(x^i))= M |\check a^i(x^i)-\tilde a^i(x^i)|$, we deduce that
\begin{align}
\label{ineq:p_f0}
& \sum_i (p^i(x^i)\cdot \check a^i(x^i) +\frac{1}{2N} |\check a^i(x^i)|^2) + f_0(\frac{1}{N} \sum_i \check  a^i(x^i))  \nonumber \\
& \geq  \sum_i (p^i(x^i)\cdot \tilde a^i (x^i)+\frac{1}{2N} |\tilde a^i(x^i)|^2) + f_0(\frac{1}{N} \sum_i \tilde  a^i(x^i))  \nonumber\\
& \qquad+ 
\frac{1}{N} \sum_i |\check a^i(x^i)-\tilde a^i(x^i)| (-C_G +M -\|Df_0\|_\infty)\nonumber \\
& \geq  \sum_i (p^i(x^i)\cdot \tilde a^i (x^i)+\frac{1}{2N} |\tilde a^i(x^i)|^2) + f_0(\frac{1}{N} \sum_i \tilde  a^i(x^i)). 
\end{align}

Thus, by summing over $(i,j)$ the inequality \eqref{ineq:hij} together with \eqref{ineq:p_f0} and by integrating over ${\bf m}\,,$ we get 
$$l({\bf \check{a}},{\bf p},{\bf m})\geq l({\bf \tilde a},{\bf p},{\bf m})\,.$$
By definition ${\bf \check{a}}\in \text{argmax}_{{\bf a}\in (L^2(\mathbb{R}^d))^N}-l({\bf a},{\bf p},{\bf m})$, we necessarily have $l({\bf \tilde a},{\bf p},{\bf m}) = l({\bf \check{a}},{\bf p},{\bf m})\,.$ By the strict convexity with respect to ${\bf a}$ of the function $l$ (because of the quadratic term and because $h_{ij}$ and $f_0$ are convex), we have ${\bf\check{a}}={\bf \tilde{a}}$ $m$-almost everywhere. 
So, we have that 
\begin{align*}
    \Vert \check{a}({\bf p,m}) \Vert_{L^\infty}=\Vert \tilde {a}({\bf p,m}) \Vert_{L^\infty}\leq C_G+ \Vert Df_0 \Vert_{\infty}\,.
\end{align*}

\end{proof}

\begin{rmk}
   Only at this stage does the specific form $h_{ij}=\tilde h_{ij} \circ |\cdot|\,,$ given in item 2 of Assumption \ref{hyp:fG}, play a role.
\end{rmk}

\section{Proof of the main Theorem}

This section is dedicated to the proof Theorem \ref{thm:diff_V_O}.\\

The idea of the proof is to compare the value functions using a form of comparison principle. We thus detail the PDEs associated with the two control problems.

\begin{enumerate}
    \item The full-information control problem :\\
    
The control problem defined in \eqref{eq:pb0} is a centralized optimal control problem.  Its value function, defined in \eqref{def:valuefx0}, satisfies 
\begin{equation}
\label{eq:V}
-\partial_t V^N -\sum_{i=1}^N \Delta_{x^i} V^N +H^N(DV^N)=0, \qquad V^N(T,\bx )= G^N(\bx)
\end{equation}
where we recall
$$
H^N((p^i)_{i=1,\dots, N}) = \sup_{{\bf a}\in (\R^d)^N} -\sum_{i=1}^N (p^i a^i +\frac{1}{2N} |a^i|^2) -f^N({\bf a})\,.
$$
\\

\item The distributed control problem :\\

The second one is the decentralized control problem \eqref{eq:pbdist}. Its value function $\mathcal V_{dist}^N:[0,T]\times (\mathcal P_2(\R^d))^N\to \R$, defined in \eqref{def:valuefxd}, (formally) solves
\begin{align}
\label{eq:Vdist}
& -\partial_t \mathcal V_{dist}^N -\sum_{i=1}^N \int \dive_y (\big (D_{m^i}\mathcal V_{dist}^N( \boldsymbol \mu^{-i},\cdot)\big )_i) \mu^i +\mathcal H^N (\big (D_{m^i} \mathcal V_{dist}^N( \boldsymbol \mu^{-i},\cdot)\big )_i, \boldsymbol \mu) =0, \nonumber \\
& \mathcal V_{dist}^N(T, \boldsymbol \mu)= \int_{(\R^d)^N} g^N({\bf x})\boldsymbol \mu(d{\bf x}), 
\end{align} 
where, for any $(q^i=q^i(x^i))_{i=1,\dots, N} \in (L^\infty(\R^d))^N$, 
$$
\mathcal H^N((q^i)_{i=1,\dots, N}, \boldsymbol \mu) = \sup_{(a^i=a^i(x^i))} -\int_{(\R^d)^N} (\sum_i (a^i(x^i) q^i(x^i)+\frac{1}{2N} |a^i(x^i)|^2) + f^N({\bf a(x)})) \boldsymbol \mu(d{\bf x}). 
$$
\end{enumerate}

As mentioned in the introduction, to compare the two problems, we will focus on the lifted version of $V^N\,, \mathcal{V}^N : [0,T] \times \mathcal{P}_2(\mathbb{R}^d)^N\to \R\,,$ defined in \eqref{def:lift_valuefx}. Note that $\mathcal{V}^N$ (formally) solves 
\begin{align*}
& -\partial_t \mathcal V^N -\sum_{i=1}^N \int \dive_y (D_{m^i}\mathcal V^N( \boldsymbol \mu^{-i},\cdot)) d\mu^i +\int_{(\R^d)^N} H^N(D V^N)d\boldsymbol \mu =0, \\
& \mathcal V^N(T, \boldsymbol \mu)=  \int_{(\R^d)^N} g^N({\bf x})\boldsymbol \mu(d{\bf x}). 
\end{align*} 

\begin{rmk}
    By the definition of $\mathcal{V}^N$, the Lions derivative $D_{m^i}\mathcal{V}^N$ can be expressed explicitly in terms of $D_iV^N\,,$ namely, for any $t \in [0,T]\,,$ ${\bf m^{-i}} \in (\mathcal P_2(\R^d))^{N-1}\,,$ and $x \in \R^d\,,$
    \begin{align*}
        D_{m^i}\mathcal{V}^N(t,{\bf m^{-i}},x)=\int_{(\R^d)^{N-1}}D_iV^N(t,{\bf y^{-i}},x){\bf m^{-i}}(d{\bf y^{-i}})\,.
    \end{align*}
\end{rmk}

\subsection{A first estimate for $\vert \mathcal{V}^N_{dist}-\mathcal{V}^N\vert$}

In order to prove Theorem \ref{thm:diff_V_O}, we introduce,
for each $\boldsymbol \mu \in \mathcal{P}_2(\R^d)^N$, the error
 
\begin{align}
    \label{def:E}
    E^N(t,\boldsymbol \mu):=  \int_{(\R^d)^N} H^N(D V^N(t,x))\boldsymbol \mu(d\bx) - \mathcal H^N((D_{m^i} \mathcal V^N(t,\cdot))_{i=1,\dots, N}, \boldsymbol \mu). 
\end{align}

For any $\boldsymbol \mu \in \mathcal{P}_2(\mathbb{R}^d)^N$, let us also introduce the process ${\bf \check{X}^{t,\boldsymbol \mu}}$, which satisfies the following McKean-Vlasov stochastic differential equation (MV-SDE), for all $i = 1, \dots, N,$
\begin{align}
    \label{def:checkX}
    d\check{X}^{t,\boldsymbol \mu,i}_{s} = \check a^{i}((D_{m^i}\mathcal{V}^N(t,{\bf m_s^{-i}},\cdot))_i,{\bf m_s})\big(s,\check{X}^{t,\boldsymbol \mu,i}_{s}\big ) ds + dW^i_s, \quad s \in (t, T), \quad \mathcal{L}(\check{X}^{t,\boldsymbol \mu,i}_{t})=\mu^i
\end{align}
with
\[
{\bf m_s} = ( \mathcal{L}(\check{X}^{t,\boldsymbol \mu,1}_{s}), \dots, \mathcal{L}(\check{X}^{t,\boldsymbol \mu,N}_{s}) )\quad  \forall s \in [t, T)\,.
\] \\
As a direct consequence of Lemma 4.4 of \cite{jackson2023approximatelyoptimaldistributedstochastic}, there exists a unique strong solution to this MV-SDE  for any $(t,\boldsymbol \mu) \in [0,T] \times \mathcal{P}_2(\mathbb{R}^d)^N$.\\

Analyzing the error $E^N$ evaluated at ${\bf m}$ will be useful to bound the difference between the two value functions. This is shown in the following lemma.

\begin{lem}
    \label{lem:diffvalfx}
    For $(t, \boldsymbol \mu) \in [0, T) \times \mathcal{P}_2(\mathbb{R}^d)^N$, we have
\begin{align*}
0 \leq \mathcal{V}^N_{\text{dist}}(t, \boldsymbol \mu) - \mathcal{V}^N(t, \boldsymbol \mu) \leq  \int_t^T E^N(s, {\bf m}_s) \, ds,
\end{align*}
where $E^N$ is given by \eqref{def:E}, and ${\bf m}_s := \left( \mathcal{L}(\check{X}_s^{t,\boldsymbol \mu,1}), \ldots, \mathcal{L}(\check{X}_s^{t,\boldsymbol \mu,N}) \right)$ is the law of the solution ${\bf \check X^{t,\boldsymbol \mu}}$ of \eqref{def:checkX}.
\end{lem}
\begin{proof} 
 First, we need to check that the lift $\mathcal{V}^N$ of $V^N$ is regular enough to apply the verification result (Proposition $3.4$) of \cite{jackson2023approximatelyoptimaldistributedstochastic}.
Using Lemma \ref{lem:D2V}, we have \begin{align*}
\sup_{t \in [0, T]} \sup_{m \in \mathcal{P}_2(\mathbb{R}^d)^N} \int_{\R^d}\left| D_y D_{m^i} \mathcal{V}^N(t, \boldsymbol \mu, x^i) \right|^2\mu^i(dx^i) < \infty, \quad i = 1, \ldots, N\,.
\end{align*}

 As mentioned before, the McKean-Vlasov SDE \eqref{def:checkX} is well-posed by Lemma 4.4 of \cite{jackson2023approximatelyoptimaldistributedstochastic}.\\
Thus, we can indeed apply Lemma 4.14 and the verification result Proposition 3.4 of \cite{jackson2023approximatelyoptimaldistributedstochastic}, with $\mathcal F(t,\boldsymbol \mu) := - E^N(t, \boldsymbol \mu)$ and $\mathcal G(\boldsymbol \mu) := \langle \boldsymbol \mu, g^N \rangle$, to get
\begin{align}
\label{eq:value_fx_E}
\mathcal{V}^N(t, \boldsymbol \mu) = \inf_{\alpha \in \mathcal{A}^{\text{dist}}} \mathbb{E} \bigg[ \int_t^T &\biggl( \frac{1}{2N} \sum_{i=1}^N\vert \alpha^i(s, X_s^i)\vert^2  +f^N({\bf \alpha}(s, X_s)) \nonumber \\
&-E^N(s, \big (\mathcal{L}(X_s^1), \ldots, \mathcal{L}(X_s^n)) \big )\biggr) ds 
+ g^N(X_T) \bigg],
\end{align}
where $X$ is given by $dX_s^i = \alpha^i(s, X_s^i) ds + dW_s^i$, with $X_t \sim \boldsymbol \mu$.\\

Moreover, we recall that ${\bf \check{X}^{t,m}}$ solves the McKean-Vlasov SDE \eqref{def:checkX} and 
$${\bf m}_s = (\mathcal{L}(\check{X}_s^{t,\boldsymbol \mu,1}), \ldots, \mathcal{L}(\check{X}_s^{t,\boldsymbol \mu,N}))\,.$$ Then, we define by $\alpha^*$ the associated control process  so that, for all $i \in \llbracket 1,N \rrbracket$ and $s \in (t,T]\,,$ $$\alpha^{*,i}_s:=\check a^{i}((D_{m^i}\mathcal{V}^N(t,{\bf m_s^{-i}},\cdot))_i,{\bf m_s})\big(s,\check{X}^{t,\boldsymbol \mu,i}_{s}\big ),$$ where ${\bf \check a}$ is given in Definition \ref{def:check_alpha}.

We deduce from the optimality criterion in Proposition 3.4 of \cite{jackson2023approximatelyoptimaldistributedstochastic} that the control $$(\alpha^{*,1},...,\alpha^{*,N}) \in \mathcal{A}^{\text{dist}}\,,$$ attains the infimum given in \eqref{eq:value_fx_E}. In particular,
\begin{align*}
\mathcal{V}^N(t, \boldsymbol \mu) &= \mathbb{E} \bigg[ \int_t^T \left( \frac{1}{2N} \sum_{i=1}^N \vert \alpha^{*,i}_s\vert^2 + f^N({\bf \alpha^{*}_s}) - E^N(s, {\bf m}_s) \right) ds + G(\check{X}_T^{t,\boldsymbol \mu}) \bigg] \\
&\geq V_{\text{dist}}^N(t, \boldsymbol \mu) - \int_t^T E^N(s, {\bf m}_s)  ds\,.
\end{align*}
\end{proof}

Therefore, to compare the two value functions $\mathcal V$ and $\mathcal V_{dist}$ and give a quantitative convergence rate, we first need to estimate $E^N$.  

We split the error into two terms: 
$$
\vert E^N \vert \leq \vert E^N_1 \vert + \vert E^N_2 \vert \,,
$$
where, for any ${\bf m} \in \mathcal{P}_2(\R^d)^N$ and $t \in [0,T]\,,$
\begin{align}
    \label{def:E1}
    E^N_1(t,{\bf m}) =   \mathcal H^N((D_{m^i} \mathcal V^N(t,\cdot))_{i=1,\dots, N}, {\bf m})-\int_{(\R^d)^N} H^N((D_{m^i} \mathcal V^N(t,x))_{i=1, \dots, N}){\bf m}(d\bx)\,,
\end{align}

and 
\begin{align}
\label{def:E2}
E^N_2(t,{\bf m}) =    \int_{(\R^d)^N} H^N((D_{m^i} \mathcal V^N(t,x))_{i=1, \dots, N}){\bf m}(d\bx) -\int_{(\R^d)^N} H^N(D V^N(t,x)){\bf m}(d\bx)   \,.
\end{align}

\subsection{Bound for $E_1^N$}
\begin{prop} 
\label{prop:bd_E1}
Let ${\bf m} \in \mathcal{P}_2(\R^d)^N$ and $t \in [0,T]$.
Then, we have  
$$
\vert E^N_1(t,{\bf m})\vert\leq K_1\,,
$$
with \begin{align*}
    K_1:= \frac{K_1'}{\sqrt{N}}(2\max_{ij}\Vert Dh_{ij}&\Vert_{\infty,C_G+\Vert Df_0\Vert_{\infty}}+\Vert Df_0\Vert_{\infty})\\
    &\qquad \times (C_G+\Vert Df_0\Vert_{\infty}) \big (\sqrt{\frac{4}{N^2}\sum_{ij}\Vert D^2h_{ij}\Vert_{\infty}^2}+\Vert D^2f_0\Vert_{\infty}\big )^2\,,
    \end{align*}
    and $K_1'\geq 0$ independent of N. \end{prop}

We start by proving the following lemma.
\begin{lem}
    \label{lem:bd_diff_f}
    Let $(a^i=a^i(x^i))_{i=1,\dots, N} \in (L^\infty(\R^d))^N$ such that $\|a^i\|_{L^{\infty}} \leq C_G+\Vert Df_0\Vert_{\infty}$ for all i and ${\bf m} \in \mathcal{P}_2(\R^d)^N$. Then, for any $i \in \llbracket 1,N \rrbracket\,,$ 
    \begin{align}
    \int_{(\R^d)^{N}}\vert \partial_if^N({\bf   a(y)})-\int_{(\R^d)^{N-1}} &\partial_if^N ({\bf   a(y^i,{\bf x^{-i}})}){\bf m}^{-i}(d\bx^{-i})\vert^2{\bf m}(d\by) \nonumber\\
    &
     \leq \frac{8(C_G+\Vert Df_0\Vert_{\infty})^2}{N^3}(\frac{4}{N}\sum_j\Vert D^2h_{ij}\Vert_{\infty}^2+\Vert D^2f_0\Vert_{\infty}^2)\,.
\end{align}
\end{lem}

\begin{proof}
    We first compute the derivatives of $f^N$. For each $i$ and ${\bf a} \in (\R^d)^N$, we have

 \begin{align*}
    \partial_if^N({\bf a})=\frac{2}{N^2}\sum_jDh_{ij}( a^i-a^j )+\frac{1}{N}Df_0(\frac{1}{N}\sum_ia^i)\,.
\end{align*}
Therefore, for each i, 
\begin{align*}
    &\vert \partial_if^N({\bf  a(y)})-\int_{(\R^d)^{N-1}} \partial_if^N ({\bf   a}(y^i,{\bf x^{-i}}){\bf m}^{-i}(d\bx^{-i})\vert^2\\
    &\leq 2\biggl|\frac{1}{N}Df_0\big (\frac{1}{N}\sum_j a^j(y^j)\big )-\frac{1}{N}\int_{(\R^d)^{N-1}} Df_0\big (  \frac{1}{N}a^i(y^i)+\frac{1}{N}\sum_{j \neq i}a^j(x^j)\big ) {\bf m^{-i}}(d\bx^{-i})\biggr|^2\\
    &+2\biggl|\frac{2}{N^2}\sum_j \biggl[Dh_{ij}(  a^i(y^i)-  a^j(y^j))
    -\int_{(\R^d)^{N-1}} Dh_{ij}(   a^i(y^i)-   a^j(x^j)) m^{j}(dx^j)\biggr]\biggr|^2\\
\end{align*}
Therefore, 
\begin{align*}
\int_{(\R^d)}&\vert \partial_if^N({\bf  a(y)})-\int_{(\R^d)^{N-1}} \partial_if^N ({\bf   a}(y^i,{\bf x^{-i}}){\bf m}^{-i}(d\bx^{-i})\vert^2{\bf m}(d\by)\\
    &\leq 2\int_{(\R^d)}\biggl|\frac{1}{N}Df_0\big (\frac{1}{N}\sum_j a^j(y^j)\big )\\
    & \quad-\frac{1}{N}\int_{(\R^d)^{N-1}} Df_0\big ( \frac{1}{N} a^i(y^i)+\frac{1}{N}\sum_{j \neq i}a^j(x^j)\big ) {\bf m^{-i}}(\bx^{-i})\biggr|^2{\bf m}(d\by)\\
    &+2\int_{(\R^d)^{N}}\biggl|\frac{2}{N^2}\sum_j \biggl[Dh_{ij}(  a^i(y^i)-  a^j(y^j))\\
    & \qquad-\int_{(\R^d)^{N-1}} Dh_{ij}(  a^i(y^i)-   a^j(x^j)) m^{j}(dx^j)\biggr]\biggr|^2{\bf m}(d\by)\,.
\end{align*}
First,
\begin{align*}
    \biggl|\frac{2}{N^2}\sum_j \biggl[Dh_{ij}&(  a^i(y^i)-  a^j(y^j))
    -\int_{(\R^d)^{N-1}} Dh_{ij}(  a^i(y^i)-   a^j(x^j)) m^{j}(dx^j)\biggr]\biggr|^2\\
    &\leq\frac{4}{N^4}\sum_{j,k}\biggl|Dh_{ij}(  a^i(y^i)-  a^j(y^j))-\int_{\R^d} Dh_{ij}(  a^i(y^i)-   a^j(x^j))m^j(dx^j)\biggr|\\
    & \qquad \times \biggl|Dh_{ik}(  a^i(y^i)-  a^k(y^k))
   -\int_{\R^d}Dh_{ik}(  a^i(y^i)-   a^k(x^k))m^k(dx^k)\biggr|\,.
\end{align*}
So, by independence of the $(   a^j)_{j \neq i}$ with respect to the product measure ${\bf m^{-i}}$, we get, for $x^i \in \R^d,$  
\begin{align}
    \label{eq:cond_indep}
    \int_{(\R^d)^{N-1}} &\biggl|\frac{2}{N^2}\sum_j \biggl[Dh_{ij}(  a^i(x^i)-  a^j(y^j))
    -\int_{(\R^d)^{N-1}} Dh_{ij}(   a^i(x^i)-   a^j(x^j))m^{j}(dx^j)\biggr]\biggr|^2{\bf m}^{-i}(d\by^{-i})\nonumber \\
    &\leq\frac{4}{N^4}\sum_{j}\int_{\R^d}\biggl|Dh_{ij}(  a^i(x^i)-  a^j(y^j))-\int_{\R^d} Dh_{ij}(  a^i(x^i)-   a^j(x^j) )m^j(dx^j)\biggr|^2m^j(dy^j)\,.
\end{align}
Moreover, by item 2. of Assumption \ref{hyp:fG}, we have for all $(x,y) \in (\R^d)^2$,
\begin{align}
    \label{eq:bddh}
    \big | Dh_{ij}(  a^i(x^i)-  a^j(y^j))-Dh_{ij}(  a^i(x^i)-  a^j(x^j))\bigl | &\leq \Vert D^2h_{ij}\Vert_{\infty} \big |   a^j(y^j)-  a^j(x^j) \big | \,.
\end{align}
Therefore, using consecutively  \eqref{eq:cond_indep} and \eqref{eq:bddh}, we have 
\begin{align}
    \label{ineq:diff_f}
    \frac{4}{N^4}\sum_{j}\int_{\R^d}&\int_{\R^d}\biggl[Dh_{ij}(  a^i(x^i)-  a^j(y^j))-\int_{\R^d} Dh_{ij}(   a^i(x^i)-   a^j(x^j) )m^j(dx^j)\biggr]^2m^j(dy^j)m^i(dx^i) \nonumber\\
    &\leq \frac{4}{N^4}\sum_j\int_{\R^d} \biggl[\int_{\R^d} \Vert D^2h_{ij}\Vert_{\infty} \big |   a^j(y^j)-  a^j(x^j) \big |m^j(dy^j)\biggr]^2m^j(dx^j) \nonumber\\
    & \leq \frac{16(C_G+\Vert Df_0\Vert_{\infty})^2}{N^4}\sum_j\Vert D^2h_{ij}\Vert_{\infty}^2\,.
\end{align}
 The last inequality comes from the boundedness of the controls (by $C_G+\Vert Df_0\Vert_{\infty}$). \\
 
 Secondly,
 \begin{align*}
     \int_{(\R^d)^N}&\biggl|\frac{1}{N}Df_0\big (\frac{1}{N}\sum_j a^j(y^j)\big )-\frac{1}{N}\int_{(\R^d)^{N-1}} Df_0\big ( \frac{1}{N} a^i(y^i)+\frac{1}{N}\sum_{j \neq i}a^j(x^j)\big ) {\bf m^{-i}}(\bx^{-i})\biggr|^2{\bf m}(d\by)\\
     &\leq \frac{\Vert D^2f_0\Vert_{\infty}^2}{N^4}\int_{(\R^d)^N}\biggl|\sum_{j\neq i}(a^j(y^j)-\int_{\R^d}a^j(x^j)m^j(dx^j))\biggr|^2{\bf m}(d\by)\\
     &\leq \frac{\Vert D^2f_0\Vert_{\infty}^2}{N^4}\\
     &\quad \times\int_{(\R^d)^N}\sum_{j,k\neq i}\biggl|a^j(y^j)-\int_{\R^d}a^j(x^j)m^j(dx^j)\biggr|\biggl|a^k(y^k)-\int_{\R^d}a^k(x^k)m^k(dx^k)\biggr|{\bf m}(d\by)\\
     &\leq \frac{\Vert D^2f_0\Vert_{\infty}^2}{N^4}\int_{(\R^d)^N}\sum_{j \neq i}\int_{\R^d}\big |a^j(y^j)-\int_{\R^d}a^j(x^j)m^j(dx^j)\big |^2m^j(dy^j)\\
     &\leq \frac{4(C_G+\Vert Df_0\Vert_{\infty})^2\Vert D^2f_0\Vert_{\infty}^2}{N^3}\,,
 \end{align*}
 where the third inequality comes from the independence of the $(a^j)_j$ with respect to the product measure ${\bf m}\,.$
 Therefore, we get 
 \begin{align}
    \label{ineq:diff_f}
    \int_{\R^d}&\int_{(\R^d)^{N-1}}\vert \partial_if({\bf   a}(x^i, {\bf y^{-i}}))-\int_{(\R^d)^{N-1}} \partial_if^N ({\bf   a}(x^i, {\bf x^{-i}})){\bf m}^{-i}(d\bx^{-i})\vert^2{\bf m}^{-i}(d\by^{-i})m^i(dx^i) \nonumber\\
    &\leq \frac{32(C_G+\Vert Df_0\Vert_{\infty})^2}{N^4}\sum_j\Vert D^2h_{ij}\Vert_{\infty}^2+ \frac{8(C_G+\Vert Df_0\Vert_{\infty})^2\Vert D^2f_0\Vert_{\infty}^2}{N^3}\nonumber\\
    &=\frac{8(C_G+\Vert Df_0\Vert_{\infty})^2}{N^3}(\frac{4}{N}\sum_j\Vert D^2h_{ij}\Vert_{\infty}^2+\Vert D^2f_0\Vert_{\infty}^2)\,.
\end{align}
\end{proof}

Before stating the next lemma, we can notice that, for all ${\bf m} \in \mathcal{P}_2(\R^d)^N, {\bf p}\in (L^\infty(\R^d))^N $ and $i \in \llbracket 1,N\rrbracket$, $(\hat a \circ p)^i$ defined in Remark \ref{rmk:hat_a_fx}, satisfies the fixed point relation:
\be\label{defhata}
(\hat a \circ p)^i(\bx) = -N p^i(x^i) - N \partial_if^N \big (((\hat a \circ p)^j(\bx))_{j=1, \dots N}\big )\quad  \forall {\bf x} \in (\R^d)^N\,.
\ee

\begin{lem}
    \label{lem:AN}
    Let ${\bf m} \in \mathcal{P}_2(\R^d)^N$ and ${\bf p}\in (L^\infty(\R^d))^N$ such that
   \begin{align*}
      \|p^i\|_{\infty}\leq \frac{C_G}{N} \; \text{ for all $i$.} 
   \end{align*} \\
    Let $\hat {\bf a} \circ {\bf p}$ be given by Definition \ref{def:hat_a_pt} and ${\bf \check a({\bf p},{\bf m})}$ by Definition \ref{def:check_alpha}.
    Then,  
    \begin{align*}
    A^N&:= \int_{(\R^d)^N} \sum_i | (\hat a\circ {\bf p})^i(\bx)- \check a({\bf p},{\bf m})^i(x^i)|^2 {\bf m}(d\bx) \\
    &\leq 8(C_G+\Vert Df_0\Vert_{\infty})^2(\frac{4}{N^2}\sum_{ij}\Vert D^2h_{ij}\Vert_{\infty}^2+\Vert D^2f_0\Vert_{\infty}^2)\,.
     \end{align*}
\end{lem}

\begin{proof}
For simplicity, let us introduce ${\bf \hat a:=\hat a \circ {\bf p}}$ from Definition \ref{def:hat_a_pt} and ${\bf \check a:=\check a}({\bf p},{\bf m})$ from Definition \ref{def:check_alpha}.\\
Thus, for each $i \in \llbracket 1,N\rrbracket$, 
 $\hat a^i$ satisfies the fixed point relation \eqref{defhata} and $\check a^i$ satisfies the relation \eqref{defchecka}.\\
   
Using these fixed point relations, we have 
\begin{align*}
&A^N:= \int_{(\R^d)^N} \sum_i | \hat a^i(\bx)- \check a^i(x^i)|^2 {\bf m}(d\bx)\\
 & = - N \int_{(\R^d)^N} \sum_i  ( \partial_if^N ({\bf\hat a(x)}) - \int_{(\R^d)^{N-1}} \partial_if^N ({\bf \check a}(x^i,{\bf y^{-i}})){\bf m}^{-i}(d\by^{-i})) \cdot (\hat a^i(\bx)- \check a^i(x^i)) {\bf m}(d\bx)\\
& = - N \int_{(\R^d)^N} \sum_i   ( \partial_if^N ({\bf \hat a(x)}) - \partial_if^N ({\bf \check a(x)})  ) \cdot (\hat a^i(\bx)- \check a^i(x^i)) {\bf m}(d\bx)\\ 
&  - N \int_{(\R^d)^N} \sum_i  \biggl (\partial_if^N ({\bf \check a(x)}) - \int_{(\R^d)^{N-1}} \partial_if^N ({\bf \check a}(x^i,{\bf y^{-i}})){\bf m}^{-i}(d\by^{-i}) \cdot  (\hat a^i(\bx)- \check a^i(x^i))\biggr) {\bf m}(d\bx)
\end{align*} 
The first term in the right-hand side is non-positive since $f^N$ is convex. Thus, by Cauchy-Schwarz,  
\begin{align*}
A^N & \leq  N \left( \int_{(\R^d)^N} \sum_i  |\partial_if^N ({\bf \check a(x)}) - \int_{(\R^d)^{N-1}} \partial_if^N ({\bf \check a}(x^i,{\bf y^{-i}})){\bf m}^{-i}(d\by^{-i}))|^2{\bf m}(d\bx)\right)^{1/2} \\
& \qquad \qquad  \times  \left(  \int_{(\R^d)^N} \sum_i |\hat a^i(\bx)- \check a^i(x^i)|^2 {\bf m}(d\bx)\right)^{1/2}\,,
\end{align*} 
which implies that 
\begin{align}\label{ineq:a_f}
A^N & \leq  N^2  \left( \int_{(\R^d)^N} \sum_i |\partial_if^N ({\bf \check a(x)}) - \int_{(\R^d)^{N-1}} \partial_if^N ({\bf \check a}(x^i,{\bf y^{-i}})){\bf m}^{-i}(d\by^{-i}))|^2{\bf m}(d\bx)\right). 
\end{align}

Moreover, by assumption,  $\Vert p^i\Vert_{\infty} \leq \frac{C_G}{N}$ for each i, so we can apply Lemma \ref{lem:check_a_bdd} for such constant $C_G$ which gives us
\be
 \| \check a^i\|_{L^\infty} \leq C_G + \Vert Df_0\Vert_{\infty} \,.
\ee

Therefore, we can now apply Lemma \ref{lem:bd_diff_f}, knowing that the controls respect the boundedness condition. Using \eqref{ineq:a_f}, we get
\begin{align}
    \label{ineq:An}
    A^N &\leq N^2 \left( \int_{(\R^d)^N} \sum_i |\partial_i f^N ({\bf \check a(x)}) - \int_{(\R^d)^{N-1}} \partial_i f^N \big ({\bf \check a}(x^i, {\bf y^{-i}})\big ){\bf m}^{-i}(d\by^{-i}))|^2{\bf m}(d\bx)\right)\nonumber\\
    &\leq N^2\sum_i\frac{8(C_G+\Vert Df_0\Vert_{\infty})^2}{N^3}(\frac{4}{N}\sum_j\Vert D^2h_{ij}\Vert_{\infty}^2+\Vert D^2f_0\Vert_{\infty}^2)\nonumber\\
    &=8(C_G+\Vert Df_0\Vert_{\infty})^2(\frac{4}{N^2}\sum_{ij}\Vert D^2h_{ij}\Vert_{\infty}^2+\Vert D^2f_0\Vert_{\infty}^2). 
\end{align}

\end{proof}
\begin{proof}[Proof of Prop. \ref{prop:bd_E1}] 

For simplicity, we will use the following notation. \\
Fix $t \in [0,T]$ and define ${\bf p}\in (L^\infty(\R^d))^N$ such that, for each i, $p^i=D_{m^i}\mathcal V^N(t,\cdot) $.

Then, for each ${\bf m} \in \mathcal{P}_2(\R^d)^N$, recall that : 
\begin{align}
  \label{def:E_1}  
E_1^N(t,{\bf m})=\int_{(\R^d)^N} H^N({(p^i({\bf{x}}))_{i=1, \dots, N})}{\bf m}(d\bx)- \mathcal H^N({(p^i({\bf{x}}))_{i=1, \dots, N})}, {\bf m}).
\end{align}

Let ${\bf \hat a:=\hat a \circ p}$ defined in Definition \ref{def:hat_a_pt}, so being optimal in the problem
\begin{align*}
\int_{(\R^d)^N} H^N(&(p^i(x^i))_{i=1, \dots, N}){\bf m}(d\bx) \\
&= \sup_{(a^i=a^i(\bx))} \int_{(\R^d)^N} -((\sum_i a^i(\bx) p^i(x^i) +\frac{1}{2N} |a^i(\bx)|^2) + f^N({\bf a(\bx)})){\bf m}(d\bx)\,,
\end{align*}
see Remark \ref{rmk:hat_a_fx}.\\

Let ${\bf \check a:=\check a(p, {\bf m})}$, as in Definition \ref{def:check_alpha}, being optimal in the definition of $\mathcal H^N((p^i)_{i=1,\dots, N}, {\bf m})$.

By definition of the Hamiltonians and the fixed point relations \eqref{defhata} and \eqref{defchecka}, we have 
\begin{align}
\label{eq:H}
\int_{(\R^d)^N} &H^N((p^i(x^i))_{i=1, \dots, N}){\bf m}(d\bx)\nonumber \\
&=  \int_{(\R^d)^N} -((\sum_i \hat a^i(\bx) p^i(x^i) +\frac{1}{2N} |\hat a^i(\bx)|^2) + f^N({\bf \hat a}(\bx))){\bf m}(d\bx) \nonumber\\ 
& = \int_{(\R^d)^N} (\frac{N}{2} \sum_i  (|p^i(x^i)|^2 + |\partial_i f^N({\bf \hat a}(\bx))|^2) + f^N({\bf \hat a}(\bx))){\bf m}(d\bx) \,,
\end{align}
while 
\begin{align}
\label{eq:Hcal}
\mathcal H^N(&(p^i)_{i=1,\dots, N}, {\bf m}) =  \int_{(\R^d)^N} -((\sum_i \check a^i(x^i) p^i(x^i) +\frac{1}{2N} |\check a^i(x^i)|^2) + f^N({\bf \check a(x)})){\bf m}(d\bx) \,.\nonumber\\ 
& = \int_{(\R^d)^N} (\frac{N}{2} \sum_i  (|p^i(x^i)|^2 + |\int_{(\R^d)^{N-1}} \partial_i f^N({\bf \check a}(x^i,{\bf y}^{-i}){\bf m}^{-i}(d\by^{-i}) |^2) + f^N({\bf \check a(x)})){\bf m}(d\bx) \,.
\end{align}

Therefore, from Definition \ref{def:E_1}, we can rewrite
\begin{align*}
\vert E^N_1(t,{\bf m}) \vert 
& = \left|  \int_{(\R^d)^N} (f^N({\bf \hat a(x)})-f^N({\bf \check a(x)})){\bf m}(d\bx) \right|\\
& +  \left|  \int_{(\R^d)^N} \frac{N}{2} \sum_i  (|\partial_i f^N({\bf \hat a(x)})|^2-|\int_{(\R^d)^{N-1}} \partial_i f^N({\bf \check a})(x^i,{\bf y}^{-i}){\bf m}^{-i}(d\by^{-i}) |^2){\bf m}(d\bx) \right|  \\ 
&  := \vert E^N_{1,1}(t,{\bf m})\vert+ \vert E^N_{1,2}(t,{\bf m})\vert .
\end{align*}

 We have, by Lemma \ref{lem:Lip_V}, that $\Vert {\bf p}\Vert_{\infty} \leq \Vert {\bf p}\Vert_{2,1}\leq \frac{C_G}{N},$ where the constant $C_G $ is defined in item 4. of Assumption \ref{hyp:fG}. So, for each i, $\Vert p^i\Vert_{\infty}\leq \Vert {\bf p}\Vert_{\infty} \leq \frac{C_G}{N}\,.$ \\
Then, we can apply Lemma \ref{lem:AN} with ${\bf m}$ and ${\bf p}=D_m\mathcal{V}^N$, and get 
\begin{align*}
    A^N&:= \int_{(\R^d)^N} \sum_i | (\hat a\circ {\bf p})^i(\bx)- \check a({\bf p},{\bf m})^i(x^i)|^2 {\bf m}(d\bx) \\
    & \leq 8(C_G+\Vert Df_0\Vert_{\infty})^2(\frac{4}{N^2}\sum_{ij}\Vert D^2h_{ij}\Vert_{\infty}^2+\Vert D^2f_0\Vert_{\infty}^2)\,.
\end{align*}

We then obtain 
\begin{align}
\label{ineq:E11}
&\vert E^N_{1,1}(t,{\bf m})\vert \nonumber\\
&= \left|  \int_{(\R^d)^N} (f^N({\bf \hat a(x)})-f^N({\bf \check a(x)})){\bf m}(d\bx) \right| \nonumber \\& \leq     \int_{(\R^d)^N} \frac{1}{N^2}\sum_{i,j}|h_{ij}( \hat a^i(\bx)- \hat a^j(\bx))-h_{ij}(\check a^i(x^i) -\check a^j(x^j))|{\bf m}(d\bx) \nonumber\\
&+\int_{(\R^d)^N}\vert f_0(\frac{1}{N}\sum_i\hat{a}^i(\bx))-f_0(\frac{1}{N}\sum_i\check{a}^i(\bx))\vert {\bf m}(d\bx) \nonumber\\
&\leq \max_{ij}\Vert Dh_{ij}\Vert_{\infty,C_G+\Vert Df_0\Vert_\infty}\int_{(\R^d)^N}\frac{1}{N^2}\sum_{i,j}\left| \hat a^i(\bx)- \hat a^j(\bx) - (\check a^i(x^i) -\check a^j(x^j)) \right|{\bf m}(d\bx) \nonumber\\
&+\Vert Df_0\Vert_\infty\int_{(\R^d)^N}\frac{1}{N}\sum_{i}\left| \hat a^i(\bx) - \check a^i(x^i) \right|{\bf m}(d\bx)\nonumber\\
& \leq (2\max_{ij}\Vert Dh_{ij}\Vert_{\infty,C_G+\Vert Df_0\Vert_\infty}+\Vert Df_0\Vert_\infty)\int_{(\R^d)^N}\frac{1}{N}\sum_{i}\left| \hat a^i(\bx) - \check a^i(x^i) \right|{\bf m}(d\bx)\nonumber\\
&\leq (2\max_{ij}\Vert Dh_{ij}\Vert_{\infty,C_G+\Vert Df_0\Vert_\infty}+\Vert Df_0\Vert_\infty) \sqrt{\int_{(\R^d)^N}\frac{1}{N}\sum_{i}\left| \hat a^i(\bx) - \check a^i(x^i) \right|^2{\bf m}(d\bx)}\nonumber\\
&=\frac{(2\max_{ij}\Vert Dh_{ij}\Vert_{\infty,C_G+\Vert Df_0\Vert_\infty}+\Vert Df_0\Vert_\infty)\sqrt{A^N}}{\sqrt{N}}\nonumber\\
&\leq \frac{(4\max_{ij}\Vert Dh_{ij}\Vert_{\infty,C_G+\Vert Df_0\Vert_\infty}+2\Vert Df_0\Vert_\infty)}{\sqrt{N}} \nonumber\\
& \qquad \times (C_G+\Vert Df_0\Vert_{\infty})\sqrt{(\frac{8}{N^2}\sum_{ij}\Vert D^2h_{ij}\Vert_{\infty}^2+2\Vert D^2f_0\Vert_{\infty}^2)}\,. 
 \end{align} 
 On the other hand 

 \begin{align*}
\vert &E^N_{1,2}(t,{\bf m}) \vert \\
&=\left|  \int_{(\R^d)^N} \frac{N}{2} \sum_i  (|\partial_i f^N({\bf \hat a(x)})|^2-|\int_{(\R^d)^{N-1}} \partial_i f^N({\bf \check a}(x^i,{\bf y}^{-i})){\bf m}^{-i}(d\by^{-i}) |^2){\bf m}(d\bx) \right|\\
&\leq \int_{(\R^d)^N}\frac{N}{2}\frac{2(2\max_{ij}\Vert Dh_{ij}\Vert_{\infty,C_G+\Vert Df_0\Vert_{\infty}}+\Vert Df_0\Vert_{\infty})}{N}\\
&\qquad \qquad\times\sum_i\left| \vert \partial_i f^N({\bf \hat a(x)})\vert -\vert \int_{(\R^d)^{N-1}} \partial_i f^N({\bf \check a}(x^i,{\bf y}^{-i})){\bf m}^{-i}(d\by^{-i})\vert  \right|\\
&\leq  \int_{(\R^d)^N} (2\max_{ij}\Vert Dh_{ij}\Vert_{\infty,C_G+\Vert Df_0\Vert_{\infty}}+\Vert Df_0\Vert_{\infty}) \sum_i \biggl(\left| \vert \partial_i f^N({\bf \hat a(x)})\vert- \vert\partial_i f^N({\bf \check a(x)})\vert\right| \\
& \qquad \qquad +\left| \vert\partial_i f^N({\bf \check a(x)}) \vert  -\int_{(\R^d)^{N-1}} \vert \partial_i f^N\big ({\bf \check a}(x^i,{\bf y}^{-i})\big ){\bf m}^{-i}(d\by^{-i})\vert\right| \biggl){\bf m}(d\bx)\,.
\end{align*}

First, using item 2. and 3. of Assumption \ref{hyp:fG}, we have
\begin{align*}
    \int_{(\R^d)^N}& \sum_i\left| \vert\partial_i f^N({\bf \hat a(x)})\vert- \vert\partial_i f^N({\bf \check a(x)})\vert\right|{\bf m}(d\bx)\\
    \leq&  \int_{(\R^d)^N}\frac{2}{N^2}\sum_{ij}\Vert D^2h_{ij} \Vert_{\infty}\vert \hat a^i(\bx) - \check a^i(x^i)-(\hat a^j(\bx) - \check a^j(x^j)) \vert{\bf m}(d\bx)\\
    + & \int_{(\R^d)^N}\frac{1}{N^2}\Vert D^2f_0\Vert_{\infty}\sum_{i}\vert \hat a^i(\bx) - \check a^i(x^i) \vert{\bf m}(d\bx)\\
    \leq& \frac{1}{N}(\sqrt{\frac{8}{N^2}\sum_{ij}\Vert D^2h_{ij} \Vert_{\infty}^2}+\frac{1}{N}\Vert D^2f_0\Vert_\infty)\sqrt{\int_{(\R^d)^N}N\sum_{i} \vert \hat a^i(\bx) - \check a^i(x^i)\vert^2{\bf m}(d\bx)}\\
    \leq& \frac{1}{N}(\sqrt{\frac{8}{N^2}\sum_{ij}\Vert D^2h_{ij} \Vert_{\infty}^2}+\frac{1}{N}\Vert D^2f_0\Vert_\infty)\sqrt{NA^N}\\
    \leq &
    \frac{1}{\sqrt{N}}(\sqrt{\frac{8}{N^2}\sum_{ij}\Vert D^2h_{ij} \Vert_{\infty}^2}+\frac{1}{N}\Vert D^2f_0\Vert_\infty)(C_G+\Vert Df_0\Vert_{\infty})\\
    &\qquad \times \sqrt{8\big (\frac{4}{N^2}\sum_{ij}\Vert D^2h_{ij}\Vert_{\infty}^2+\Vert D^2f_0\Vert_{\infty}^2\big )}\,.
\end{align*}

Secondly, by Cauchy-Schwarz inequality and by using Lemma \ref{lem:bd_diff_f},
\begin{align*}
    \int_{(\R^d)^N}&\sum_i 
\left| \vert  \partial_i f^N({\bf \check a(x)}) \vert   - \vert \int_{(\R^d)^{N-1}} \partial_i f^N({\bf \check a}(x^i,{\bf y}^{-i})){\bf m}^{-i}(d\by^{-i}) \vert \right|{\bf m}(d\bx) \\
&\leq \sqrt{N}\sqrt{\int_{(\R^d)^N}\sum_i
\left| \vert \partial_i f^N({\bf \check a(x)})  \vert -\vert \int_{(\R^d)^{N-1}} \partial_i f^N({\bf \check a}(x^i,{\bf y}^{-i})){\bf m}^{-i}(d\by^{-i}) \vert \right|^2 {\bf m}(d\bx) }\\
&\leq \sqrt{N}\sqrt{\frac{8(C_G+\Vert Df_0\Vert_{\infty})^2}{N^2}(\frac{4}{N^2}\sum_{ij}\Vert D^2h_{ij}\Vert_{\infty}^2+\Vert D^2f_0\Vert_{\infty}^2)}\\
&\leq \frac{1}{\sqrt{N}}(C_G+\Vert Df_0\Vert_{\infty})\sqrt{8(\frac{4}{N^2}\sum_{ij}\Vert D^2h_{ij}\Vert_{\infty}^2+\Vert D^2f_0\Vert_{\infty}^2)}\,,
\end{align*}
where the second inequality comes from Lemma \ref{lem:bd_diff_f}.\\
Thus, 
\begin{align}
\label{ineq:E12}
    \vert& E^N_{1,2}(t,{\bf m})\vert \nonumber \\
    &\leq  \int_{(\R^d)^N} (2\max_{ij}\Vert Dh_{ij}\Vert_{\infty,C_G+\Vert Df_0\Vert_{\infty}}+\Vert Df_0\Vert_{\infty})  \sum_i \biggl(\left| \vert \partial_i f^N({\bf \hat a(x)})\vert - \vert \partial_i f^N({\bf \check a(x)})\vert \right| \nonumber\\
    & \qquad +\left| \vert \partial_i f^N({\bf \check a(x)}) \vert  -\vert \int_{(\R^d)^{N-1}} \partial_i f^N({\bf \check a}(x^i,{\bf y}^{-i})){\bf m}^{-i}(d\by^{-i})\vert  \right| \biggr){\bf m}(d\bx) \nonumber \\
    & \leq (2\max_{ij}\Vert Dh_{ij}\Vert_{\infty,C_G+\Vert Df_0\Vert_{\infty}}+\Vert Df_0\Vert_{\infty})\frac{1}{\sqrt{N}}(C_G+\Vert Df_0\Vert_{\infty})\nonumber\\
    & \qquad \times \sqrt{8 \big (\frac{4}{N^2}\sum_{ij}\Vert D^2h_{ij}\Vert_{\infty}^2+\Vert D^2f_0\Vert_{\infty}^2\big )} (1+\sqrt{\frac{8}{N^2}\sum_{ij}\Vert D^2h_{ij} \Vert_{\infty}^2}+\frac{1}{N}\Vert D^2f_0\Vert_\infty)\,.
\end{align}
By \eqref{ineq:E11} and \eqref{ineq:E12}, this shows that
$$
E^N_1(t,{\bf m})\leq K_1\,,
$$
with \begin{align*}
    K_1:= \frac{K_1'}{\sqrt{N}}(2\max_{ij}\Vert Dh_{ij}\Vert&_{\infty,C_G+\Vert Df_0\Vert_{\infty}}+\Vert Df_0\Vert_{\infty})\\
    &\qquad \times (C_G+\Vert Df_0\Vert_{\infty}) \big (\sqrt{\frac{4}{N^2}\sum_{ij}\Vert D^2h_{ij}\Vert_{\infty}^2}+\Vert D^2f_0\Vert_{\infty}\big )^2\,,
    \end{align*}
    and $K_1'\geq 0$ independent of N. 
\end{proof}
\subsection{Bound for $E_2^N$}
We now bound the error $$E_2^N(t,{\bf m}):=  \int_{(\R^d)^N} H^N(D V^N(t,{\bf x})){\bf m}(d\bx)-  \int_{(\R^d)^N} H^N((D_{m^i} \mathcal V^N(t,\cdot))_{i=1, \dots, N}){\bf m}(d\bx)\,.$$

\begin{prop}
    \label{prop:E1/EQ}
    For all ${\bf m} \in \mathcal{P}_2(\R^d)^N$ and $t \in [0,T]$, $E_2^N(t,{\bf m})$ satisfies $$-K_1 \leq E_2^N(t,{\bf m}) \leq E_Q^N(t,{\bf{m}})+2(C_G +\Vert Df_0\Vert_\infty)\sqrt{E_Q^N(t,{\bf{m}})}\,,$$
    where with \begin{align*}
    K_1:= \frac{K_1'}{\sqrt{N}}(2\max_{ij}&\Vert Dh_{ij}\Vert_{\infty,C_G+\Vert Df_0\Vert_{\infty}}+\Vert Df_0\Vert_{\infty})\\
    &\qquad \times (C_G+\Vert Df_0\Vert_{\infty}) \big (\sqrt{\frac{4}{N^2}\sum_{ij}\Vert D^2h_{ij}\Vert_{\infty}^2}+\Vert D^2f_0\Vert_{\infty}\big )^2\,,
    \end{align*}
    $K_1'\geq 0$ independent of N and
    $$
    E_Q^N(t,{\bf m}):=N\sum_i \left[\int_{(\R^d)^N}\vert D_iV^N(t,\bx)\vert ^2{\bf m}(d\bx)-\int_{(\R^d)^N}\vert D_{m^i}\mathcal{V}^N(t,{\bf m}^{-i},x^i)\vert ^2 m^i(dx^i)\right]\,.
    $$
\end{prop}

   
\begin{proof}[Proof of Prop. \ref{prop:E1/EQ}]

Let ${\bf m} \in \mathcal{P}_2(\R^d)^N$ and $t \in [0,T]$. For the sake of simplicity,
define for each $i \in \llbracket1,N \rrbracket,$ 
\begin{align}
    \label{def:pi}
    p^i:=D_iV^N(t,\cdot),
\end{align} and
\begin{align}
    \label{def:qi}
    q^i:=\int_{(\R^d)^{N-1}}p^i(\boldsymbol{x}){\bf m}^{-i}(d\boldsymbol{x^{-i}})=D_{m^i} \mathcal V^N(t,{\bf m}^{-i},\cdot)\,.
\end{align}

By Lemma \ref{lem:Lip_V}, we see that for all $i \in \llbracket1,N \rrbracket,$
$$\Vert p^i \Vert_{\infty}\leq \frac{C_G}{N}\,. $$
We first prove that $ E_2^N(t,{\bf m})\geq- K_1$.\\

Recalling Definition \ref{def:check_alpha}, let ${\bf \check a(q)}:={\bf \check a(q,{\bf m})}$ be the optimal solution associated to the problem $$\mathcal H^N((q^i)_{i=1,\dots, N}, {\bf m})
= \sup_{(a^i=a^i(x^i))} -\int_{(\R^d)^N} (\sum_i (a^i(x^i) q^i(x^i) +\frac{1}{2N} |a^i(x^i)|^2) + f^N({\bf a(x)})) {\bf m}(d{\bf x}). 
$$

Thus,
\begin{align*}
    \int_{(\R^d)^N}&H^N(\boldsymbol{p(x)}){\bf m}(d\boldsymbol{x}) \\
    &\geq \int_{(\R^d)^N}\biggl(-\sum_i\check a^i({\bf q})(x^i)p^i(\boldsymbol{x})+\frac{1}{2N}|\check a^i({\bf q})(x^i)|^2-f^N({\bf \check a(q)(x)})\biggr){\bf m}(d\boldsymbol{x})\\
    & = \int_{\R^d}\biggl(-\sum_i\check a^i({\bf q})(x^i)\int_{(\R^d)^{N-1}}p^i(\boldsymbol{x}){\bf m}^{-i}(d\boldsymbol{x^{-i}})+\frac{1}{2N}|\check a^i({\bf q})(x^i)|^2\biggr)m^i(dx^i)\\
& \qquad -\int_{(\R^d)^N}f^N({\bf \check a(q)(x)}){\bf m}(d\boldsymbol{x}) \\
    &= \int_{(\R^d)^N}\biggl(-\sum_i\check a^i({\bf q})(x^i)q^i(x^i)+\frac{1}{2N}|\check a^i({\bf q})(x^i)|^2-f^N({\bf \check a(q)(x)})\biggr){\bf m}(d\boldsymbol{x})\\
    &= \mathcal{H}^N(\boldsymbol{q}, {\bf m})\\
    &\geq \int_{(\R^d)^N}H^N(\boldsymbol{q}(x)){\bf m}(d\boldsymbol{x}) - E^N_1(t,{\bf m})\,.
\end{align*}
We use Prop. \ref{prop:bd_E1} to conclude that

$$
\int_{(\R^d)^N}H^N(\boldsymbol{p}){\bf m}(d\boldsymbol{x})-\int_{(\R^d)^N}H^N(\boldsymbol{q}){\bf m}(d\boldsymbol{x})\geq -K_1\,.
$$
We now want to prove 
\begin{align*}
     E_2^N(t,{\bf m}) \leq E_Q^N(t,{\bf{m}})+2(C_G+\Vert Df_0\Vert_{\infty})\sqrt{E_Q^N(t,{\bf{m}})}\,.
\end{align*}

    Define  \[ \phi := \begin{cases}
(\R^d)^N \to \mathbb{R^d}\\
{\bf y} \mapsto N|{\bf y}|^2-  H^N({\bf y})\,,
\end{cases}
\]

and let's show that it is convex. \\

First notice that, by the envelope theorem, for all $y \in (\R^d)^N$, $D_{y^i} \phi({\bf y})= 2Ny^i+ \hat a^i({\bf y})$. \\

Let ${\bf y},{\bf \tilde y} \in (\R^d)^N$,
\begin{align*}
    \langle \nabla \phi({\bf y})-\nabla \phi({\bf \tilde y}),{\bf y} -{\bf \tilde y}\rangle&=\sum_i \biggl[2N(y^i -\tilde y^i)+\hat a^i(y)-\hat a^i(\tilde y)\biggr]\cdot\biggl[y^i -\tilde y^i \biggr]\\
    &= \sum_i 2N\vert y^i -\tilde y^i\vert^2+\sum_i(\hat a^i(y)-\hat a^i(\tilde y))(y^i -\tilde y^i)\\
    &\geq \sum_i 2N\vert y^i -\tilde y^i \vert^2-\sqrt{N^2\sum_i\vert y^i -\tilde y^i\vert^2}\sqrt{\sum_i\vert y^i -\tilde y^i\vert ^2}\\
    &=\sum_i 2N\vert y^i -\tilde y^i\vert ^2-\sum_i N\vert y^i -\tilde y^i \vert ^2\\
    &\geq 0\,.
\end{align*}
where $\hat a(.)$ is defined in \ref{def:hat_a_pt} and where the first inequality comes from Lemma \ref{lem:bd_diff_a}.\\


Recalling the definition of ${\bf p}$ and ${\bf q}$ respectively in \eqref{def:pi} and \eqref{def:qi}, we deduce that
\begin{align*}
 &\int_{(\R^d)^N}\phi({\bf p(x)}){\bf m}(d\bx)  \\
 &\geq \int_{(\R^d)^N}\phi({\bf q(x)}{\bf m}(d\bx)  +\sum_i  \int_{(\R^d)^N}\big (2Nq^i(x^i)+\hat a^i({\bf q(x)})\big )\cdot (p^i({\bf x})-q^i(x^i)) {\bf m}(d\bx)  \,.
\end{align*}
First, one can notice that $ \int_{(\R^d)^N} 2N q^i(x^i)\cdot (p^i(\bx)-q^i(x^i)) {\bf m}(d\bx)=0$.\\

Secondly, 
\begin{align*}
  \sum_i  \int_{(\R^d)^N}&\hat a^i({\bf q}(\bx))\cdot (p^i(\bx)-q^i(x^i)) {\bf m}(d\bx) \\
 &\geq  \sum_i  \int_{(\R^d)^N}(\hat a^i({\bf q}(\bx))-\check a^i(q^i(x^i))+\check a^i(q^i(x^i)))\cdot (p^i(\bx)-q^i(x^i)) {\bf m}(d\bx) \\
   & =  \sum_i  \int_{(\R^d)^N}(\hat a^i({\bf q}(\bx))-\check a^i(q^i(x^i)))\cdot (p^i(\bx)-q^i(x^i)) {\bf m}(d\bx) \\
   &\geq -2(C_G +\Vert Df_0\Vert_\infty) \sqrt{  N\int_{(\R^d)^N}\sum_i \vert q^i(x^i)-p^i(\bx)\vert^2 {\bf m}(d\bx)}\,,
\end{align*}
because $\int_{(\R^d)^N}\check a^i(q^i(x^i))\cdot (q^i(x^i)-p^i(\bx)) {\bf m}(\bx)=0$. The last inequality comes from Cauchy Schwarz inequality and from the boundedness of $\sum_i  \int_{(\R^d)^N}(\hat a^i({\bf q}(\bx))-\check a^i(q^i(x^i)))^2$ according to Lemma \ref{lem:AN}. \\

We thus have 
\begin{align*}
 \int_{(\R^d)^N}\phi({\bf p}(\bx)){\bf m}(\bx)&\geq \int_{(\R^d)^N}\phi({\bf q}(\bx)){\bf m}(\bx)  \\
& \qquad -2(C_G +\Vert Df_0\Vert_\infty) \sqrt{  N\int_{(\R^d)^N}\sum_i \vert q^i(x^i)-p^i(\bx)\vert^2 {\bf m}(\bx)} . 
\end{align*}
Therefore, with $$\int_{(\R^d)^N}\phi({\bf p}(\bx)){\bf m}(\bx)= \int_{(\R^d)^N}N\vert {\bf p}(\bx)\vert^2{\bf m}(\bx)-\int_{(\R^d)^N}H^N({\bf p}(\bx)){\bf m}(\bx),$$ we get, 
\begin{align*}
    E_2^N(t,{\bf{m}}) \leq E_Q^N(t,{\bf{m}})+2(C_G +\Vert Df_0\Vert_\infty)\sqrt{E_Q^N(t,{\bf{m}})}\,.
\end{align*}

\end{proof}

\subsection{Dynamics of $E_Q^N$ along a certain curve}

Recall that the definition of $E^N_Q$ is given in Proposition \ref{prop:E1/EQ}.\\
According to Lemma \ref{lem:diffvalfx}, we do not need to find a uniform bound for the error $E^N$ but we will only focus on $E^N$ applied to a specific distribution, namely $({\bf m_t})_{t\in[0,T]}=(\mathcal{L}(\check X^{\boldsymbol \mu,1}_t),...,\mathcal{L}(\check X^{\boldsymbol \mu,N}_t)\big )_{t\in[0,T]}$, where ${\bf \check X}$ is the process defined by the SDE \eqref{def:checkX}. In fact, by Proposition \ref{prop:E1/EQ}, it suffices to focus on $E^N_Q$ applied to $({\bf m_t})_{t\in[0,T]}\,.$\\

Let us introduce some notations. \\

Denote by $\check a^{\mathcal V}$ the control function of the form
   \begin{align}
       \label{def:checkaVcal}
       \check a^{\mathcal V}  : \begin{cases}
[0,T] \times \mathcal{P}_2(\R^d)^N \times (\R^d)^N \to (\mathbb{R}^d)^N\\
(t,{\bf m},\bx) \mapsto \check{a}\big ((D_{m^i}\mathcal{V}^N(t,{\bf m}^{-i},\cdot))_i,{\bf m}\big )(\bx),
\end{cases}
   \end{align}
   where $\check a$ is given in Definition \ref{def:check_alpha}.\\
   
Denote also by $\hat a^{\mathcal{V}}$ and $\hat a^{V}$, the control functions of the form
   \begin{align}
   \label{def:aVcal}
       \hat a^{\mathcal{V}}  :\begin{cases}
[0,T] \times \mathcal{P}_2(\R^d)^N \times (\R^d)^N \to (\mathbb{R}^d)^N\\
(t,{\bf m},\bx) \mapsto \hat{a} \circ\big((D_{m^i}\mathcal{V}^N(t,{\bf m^{-i}},\cdot))_i\big )(\bx),
\end{cases}
   \end{align}

   \begin{align}
       \label{def:aV}
       \hat a^{V}  : \begin{cases}
[0,T] \times (\R^d)^N \to (\mathbb{R}^d)^N\\
(t,\bx) \mapsto \hat{a} \circ \big (DV^N(t,\cdot)\big )(\bx)\,.
\end{cases}
   \end{align}

where $\hat{a}$ is given in Definition \ref{def:hat_a_pt}.\\

As mentioned above, studying the dynamics of $E_Q^N(s,{\bf m_s})$ will be sufficient to bound $E^N(s,{\bf m_s})$, for each $s \in [0,T]$.

\begin{prop}
    \label{prop:EQ_mT}
    Given $\boldsymbol \mu \in \mathcal{P}_2(\mathbb{R}^d)^N$ satisfying the Poincaré inequality with some constant $c_p$ and $t \in [0,T]$, consider ${\bf \check X}^{t,\boldsymbol \mu}$ solving the previous McKean-Vlasov SDE \eqref{def:checkX}.
Then, for all $s \in [t,T],$ 
\begin{align}
    \label{ineq:EQ}
E_Q^N(s,{\bf m_s}) 
    &\leq e^{3\frac{C_G}{N}(T-s)}E_Q^N(T,{\bf m_T}) \nonumber\\
    & \qquad+ \frac{(C_G+\Vert Df_0\Vert_{\infty})^2}{N}(\frac{4}{N^2}\sum_{ij}\Vert D^2h_{ij}\Vert_{\infty}^2+\Vert D^2f_0\Vert_{\infty}^2)\big ( e^{3\frac{C_G}{N}(T-s)}-1
    \big )
    \,. 
    \end{align}
where $$E_Q^N(T,{\bf m_T})=C_pN \mathbb{E} \sum_{ij} |D_{ij} g^N({\bf \check X}^{t,\boldsymbol \mu}_T)|^2,$$ 
 \text{ and}
$$C_p:=\frac{e^{2C_G T} - 1}{2C_G} + c_p e^{2C_G T}\,.$$

\end{prop}

The proof of the proposition requires several preliminary steps.

\begin{lem}
    \label{lem:Lip_aVcal}
    Given $t \in [0,T]$ and ${\bf m} \in \mathcal{P}_2(\mathbb{R}^d)^N$, let $\bx, \bar{\bx} \in (\R^d)^N$. Then, 
\begin{align*}
    \Vert \check a^{\mathcal{V}}(t,{\bf m},\bx)-\check a^{\mathcal{V}}(t,{\bf m},\bar{\bx}) \Vert^2\leq C_G^2\Vert \bx-\bar \bx \Vert^2\,.
\end{align*}

    In other words, $\check a^{\mathcal{V}}$ is $C_G-$Lipschitz with respect to its third variable $\bx$.
\end{lem}
\begin{proof}
Let $t \in [0,T]$, ${\bf m} \in \mathcal{P}_2(\mathbb{R}^d)^N$ and $\bx, \bar{\bx} \in (\R^d)^N$. By Lemma \ref{lem:D2V}, $V^N$ is twice differentiable and $D^2V^N \leq \frac{C_G}{N} I_{nd\times nd}$, so $DV^N$ is $\frac{C_G}{N}$-Lipschitz.
Therefore,
\begin{align}
    \label{ineq:b_lip}
    \Vert \check a^{\mathcal{V}}&(t,{\bf m},\bx)-\check a^{\mathcal{V}}(t,{\bf m},\bar{\bx}) \Vert^2 \nonumber\\
    &=\sum_i \vert\check a^i\big ((D_{m^i}\mathcal{V}^N(t,{\bf m}^{-i},\cdot))_i,{\bf m}\big )(x^i)-\check a^i\big ((D_{m^i}\mathcal{V}^N(t,{\bf m}^{-i},\cdot))_i,{\bf m}\big )(\bar x^i)\vert^2\nonumber \\
    & \leq N^2\sum_i\vert D_{m^i}\mathcal{V}^N(t,{\bf m}^{-i},x^i)-D_{m^i}\mathcal{V}^N(t,{\bf m}^{-i},\bar x^i))\vert^2\nonumber \\
    &\leq N^2\sum_i\biggl|\int_{(\R^d)^{N-1}}D_iV^N(t,{\bf y^{-i}},x^i)-D_iV^N(t,{\bf y^{-i}},\bar x^i){\bf m}^{-i}(d{\bf {\bf y^{-i}}})\biggr|^2\nonumber \\
    &\leq N^2\sum_i\int_{(\R^d)^{N-1}}\vert D_iV^N(t,{\bf y^{-i}},x^i)-D_iV^N(t,{\bf y^{-i}},\bar x^i)\vert^2{\bf m}^{-i}(d{\bf {\bf y^{-i}}})\nonumber \\
    & \leq N^2\frac{C_G^2}{N^2}\sum_i\vert x^i-\bar x^i\vert^2\nonumber \\
    &=C_G^2\Vert \bx-\bar \bx \Vert^2\,,
\end{align}
where the first inequality comes from Lemma \ref{lem:bd_diff_check_a}, the second from the definition of the lift of $V$ and the last inequality from Lemma \ref{lem:D2V}.
\end{proof}

\begin{lem}
 \label{ineq:b_holder}
    For all $\bx, \bar{\bx} \in (\R^d)^N, t \in [0,T]$, and ${\bf m} \in \mathcal{P}_2(\mathbb{R}^d)^N$, we have
    $$\big (\check a^{\mathcal{V}}(t,{\bf m},\bar{\bx})-\check a^{\mathcal{V}}(t,{\bf m},\bx) \big )\cdot(\bx-\bar \bx)\leq C_G\sum_i\vert x^i-\bar x^i\vert ^2\,.$$
\end{lem}

\begin{proof}
The proof is direct using Lemma \ref{lem:Lip_aVcal} and applying Cauchy-Schwarz inequality.
\end{proof}
We can now prove the previous proposition.

\begin{proof}[Proof of Prop. \ref{prop:EQ_mT}]

This proof closely follows the proof of Theorem $4.5$ in \cite{jackson2023approximatelyoptimaldistributedstochastic}.\\

  Let $\boldsymbol \mu \in \mathcal{P}_2(\mathbb{R}^d)^N$ satisfying the Poincaré inequality with some constant $c_p$ and $t \in [0,T]$, and introduce ${\bf \check X}:={\bf \check X}^{t,\boldsymbol \mu}$ solution to the SDE \eqref{def:checkX}, namely, for all $i = 1, \dots, N,$
\begin{align*}
    d\check{X}^{t,\boldsymbol \mu,i}_{s} = \check a^{i}((D_{m^i}\mathcal{V}^N(t,{\bf m_s^{-i}},\cdot))_i,{\bf m_s})\big(s,\check{X}^{t,\boldsymbol \mu,i}_{s}\big ) ds + dW^i_s, \quad s \in (t, T), \quad \mathcal{L}(\check{X}^{t,\boldsymbol \mu,i}_{t})=\mu^i
\end{align*}
with
\[
{\bf m_s} = ( \mathcal{L}(\check{X}^{t,\boldsymbol \mu,1}_{s}), \dots, \mathcal{L}(\check{X}^{t,\boldsymbol \mu,N}_{s}) )\quad  \forall s \in (t, T)\,.
\] \\
The strategy used in \cite{jackson2023approximatelyoptimaldistributedstochastic} is to compute the differential of the Ito process 
$$|D_i V^N (s, {\bf \check{X}_s})|^2 - |D_{m^i} \mathcal V^N (s, {\bf m^{-i}_s},\check{X}^i_{s})|^2=
|D_i V^N (s, {\bf \check{X}_s})|^2 - |\E[D_i V^N (s, {\bf \check{X}_s}) | \check{X}^i_{s}]|^2\,.
$$
Similarly, by differentiating the term above, using the Ito formula and recalling the definitions of \( \check a^{\mathcal{V}} \) and \( \hat a^{V} \), we obtain the following:

\begin{align*}
    \frac{dE_Q^N}{dt}(s,m_s)=A_1+A_2\,,
\end{align*}
with 
\begin{align*}
    \left\{
    \begin{array}{l}
    A_1 := \E\biggl[ \sum_{i,j}\biggl(D_iV^N(s,{\bf \check{X}_s})-\E\big(D_iV^N(s,{\bf \check{X}_s}) \vert \check{X}_s^i \big )\biggr)^T\\
    \qquad \qquad \qquad D_{ij}V^N(s,{\bf \check{X}_s})\biggl( \check a^{\mathcal{V}}_j (s, {\bf m_s},{\bf \check{X}_{s}}) - \hat a^{V}_j (s, {\bf \check{X}_{s}})\biggr)\biggr]\,,\\
    A_2 := \frac{1}{2}\E\biggl[\sum_{i,j}\vert D_{ij}V^N(s,{\bf \check{X}_s})\vert^2 -\sum_i\vert \E(D_{ii}V^N(s,{\bf \check{X}_s})\vert \check{X}_s^i)\vert^2\biggr]\geq 0\,.
    \end{array}
    \right .
\end{align*}
We now focus on the term $A_1$. \\
First, recall that the operator norm is defined such that, for any $A \in \R^{Nd\times Nd}$: 
$$
\Vert A \Vert_{op}:= \sup_{\|v\|=1} \|A v\|\,.
$$
We decompose $A_1$ into two terms. 
\begin{align*}
    A_1 =A_{1,1}+A_{1,2}&\\
    :=\E\biggl[ \sum_{i,j}\biggl(D_i&V^N(s,{\bf \check{X}_s})-\E\big(D_iV^N(s,{\bf \check{X}_s}) \vert \check{X}_s^i \big )\biggr)^T\\
    & \qquad D_{ij}V^N(s,{\bf \check{X}_s})\biggl( \check a^{\mathcal{V}}_j (s, {\bf m_s},{\bf \check{X}_{s}})  -\hat a^{\mathcal{V}}_j (s, {\bf m_s},{\bf \check{X}_{s}})\biggr)\biggr]\\
     \quad + \E\biggl[ \sum_{i,j}\biggl(&D_iV^N(s,{\bf \check{X}_s})-\E\big(D_iV^N(s,{\bf \check{X}_s}) \vert \check{X}_s^i \big )\biggr)^T\\
     & \qquad D_{ij}V^N(s,{\bf \check{X}_s})\biggl( \hat a^{\mathcal{V}}_j (s, {\bf m_s},{\bf \check{X}_{s}})- \hat a^{V}_j (s, {\bf \check{X}_{s}})\biggr)\biggr]
\end{align*}

First,
\begin{align*}
    A_{1,1}=\biggl|& \E\biggl[ \sum_{i,j}\biggl(D_iV^N(s,{\bf \check{X}_{s}})-\E\big(D_iV^N(s,{\bf \check{X}_{s}}) \vert \check{X}_s^i \big )\biggr)^T\\
    & \qquad  D_{ij}V^N(s,{\bf \check{X}_{s}})\biggl(\check a^{\mathcal{V}}_j (s,{\bf m_s}, {\bf \check{X}_{s}}) - \hat a^{\mathcal V}_j (s,{\bf m_s}, {\bf \check{X}_{s}})\biggr)\biggr]\biggr|\\
     \leq& \Vert D^2 V^N\Vert_{op}  \sqrt{\E\sum_{i} \vert D_iV^N(s,{\bf \check{X}_{s}})-\E\big(D_iV^N(s,{\bf \check{X}_{s}}) \vert \check{X}_s^i \big )\vert^2} \\ 
     &\qquad \sqrt{\E\sum_{j} \vert \check a^{\mathcal{V}}_j (s,{\bf m_s}, {\bf \check{X}_{s}}) - \hat a^{\mathcal V}_j (s,{\bf m_s}, {\bf \check{X}_{s}})\vert^2}\\
       \leq& \frac{C_G}{N^{\frac{3}{2}}}\sqrt{E_Q^N(s,m_s)}\sqrt{8}(C_G+\Vert Df_0\Vert_{\infty})\sqrt{(\frac{4}{N^2}\sum_{ij}\Vert D^2h_{ij}\Vert_{\infty}^2+\Vert D^2f_0\Vert_{\infty}^2)}\,,
\end{align*}
where the first inequality comes from Cauchy-Schwarz inequality and the last inequality comes from Lemma \ref{lem:AN} which can be applied because, by Lemma \ref{lem:Lip_V}, that for each i, $\Vert D_{m^i}\mathcal V^N\Vert_{\infty} \leq \frac{C_G}{N}$ (where the constant $C_G$ is defined in item 4. of Assumption \ref{hyp:fG}).\\

Secondly, by Lemma \ref{lem:bd_diff_a}, we get
\begin{align}
    \label{ineq:a_p}
    \sum_j\vert\hat a^{\mathcal{V}}_j (s,{\bf m_s}, {\bf \check{X}_{s}}) - \hat a^{V}_j (s, {\bf \check{X}_{s}}) \vert^2\leq N^2\sum_j\vert D_{m^j} \mathcal V^N(s,{\bf m^{-j}_s},\check{X}_{s}^j)-D_jV(s,{\bf \check{X}_{s}})\vert^2 \,.
\end{align}
Then,
\begin{align*}
      A_{1,2}=\biggl| \E\biggl[ \sum_{i,j}\biggl(D_i&V^N(s,{\bf m_s},{\bf \check{X}_{s}})-\E\big(D_iV^N(s,{\bf \check{X}_{s}}) \vert \check{X}_s^i \big )\biggr)^T \\
      & \qquad \qquad D_{ij}V^N(s,{\bf \check{X}_{s}})\biggl(\hat a^{\mathcal{V}}_j (s,{\bf m_s}, {\bf \check{X}_{s}}) - \hat a^{V}_j (s, {\bf \check{X}_{s}})\biggr)\biggr]\biggr|\\
     \leq \Vert D^2 V^N\Vert_{op}& \E\biggl[ \sqrt{\sum_{i} \vert D_iV^N(s,{\bf \check{X}_{s}})-\E\big(D_iV^N(s,{\bf \check{X}_{s}}) \vert \check{X}_s^i \big )\vert^2} \\
     &\qquad \qquad \sqrt{\sum_{j} \vert \hat a^{\mathcal{V}}_j (s, {\bf m_s},{\bf \check{X}_{s}}) - \hat a^{V}_j (s, {\bf \check{X}_{s}})\vert^2}\biggr]\\
     \leq \Vert D^2 V^N\Vert_{op}&  \sqrt{\E\bigg[\sum_{i} \vert D_iV^N(s,{\bf \check{X}_{s}})-\E\big(D_iV^N(s,{\bf \check{X}_{s}}) \vert \check{X}_s^i \big )\vert^2\bigg]} \\
     &N\sqrt{\E\bigg[\sum_{i} \vert D_iV^N(s,{\bf \check{X}_{s}})-\E\big(D_iV^N(s,{\bf \check{X}_{s}}) \vert \check{X}_s^i \vert\bigr )^2\bigg]}\\
     \leq \Vert D^2 V^N\Vert_{op}&N \E\biggl[ \sum_{i} \vert D_iV^N(s,{\bf \check{X}_{s}})-\E\big(D_iV^N(s,{\bf \check{X}_{s}}) \vert \check{X}_s^i \big )\vert^2\biggr]\,.
     \end{align*}
     Therefore, by Lemma \ref{lem:D2V},
     \begin{align*}
         A_{1,2} \leq \frac{C_G}{N}N\E\biggl[ &\sum_{i} \vert D_iV^N(s,{\bf \check{X}_{s}})-\E\big(D_iV^N(s,{\bf \check{X}_{s}}) \vert \check{X}_s^i \big )\vert^2\biggr]=\frac{C_G}{N}E_Q^N(s,m_s)\,.
     \end{align*}
Therefore, summing the two previous results,
\begin{align*}
    \vert A_1 \vert \leq  \frac{C_G}{N}E_Q^N(s,{\bf m_s}) +\frac{C_G}{N^{\frac{3}{2}}}(C_G+\Vert Df_0\Vert_{\infty})\sqrt{8(\frac{4}{N^2}\sum_{ij}\Vert D^2h_{ij}\Vert_{\infty}^2+\Vert D^2f_0\Vert_{\infty}^2)}\sqrt{E_Q^N(s,m_s)}\,.
\end{align*}

We then find that 
\begin{align*}
    &\frac{dE_Q^N}{ds}(s,{\bf m_s})\geq -\frac{C_G}{N}E_Q^N(s,{\bf m_s})\\
    &\qquad -\sqrt{2}\sqrt{2\frac{C_G}{N^2}(C_G+\Vert Df_0\Vert_{\infty})^2(\frac{4}{N^2}\sum_{ij}\Vert D^2h_{ij}\Vert_{\infty}^2+\Vert D^2f_0\Vert_{\infty}^2)}\sqrt{2\frac{C_G}{N}E_Q^N(s,{\bf m_s})}\,.
\end{align*}
By Young's Inequality, 
\begin{align*}
&\sqrt{2\frac{C_G}{N^2}(C_G+\Vert Df_0\Vert_{\infty})^2(\frac{4}{N^2}\sum_{ij}\Vert D^2h_{ij}\Vert_{\infty}^2+\Vert D^2f_0\Vert_{\infty}^2)}\sqrt{2\frac{C_G}{N}E_Q^N(s,{\bf m_s})}\\
&\leq \frac{C_G}{N^2}(C_G+\Vert Df_0\Vert_{\infty})^2(\frac{4}{N^2}\sum_{ij}\Vert D^2h_{ij}\Vert_{\infty}^2+\Vert D^2f_0\Vert_{\infty}^2)+\frac{C_G}{N}E_Q^N(s,{\bf m_s})\,.
\end{align*}

Therefore, 
\begin{align*}
    &\frac{dE_Q^N}{ds}(s,{\bf m_s})\\
    &\geq -\frac{(1+\sqrt{2})C_G}{N}E_Q^N(s,{\bf m_s})-\frac{2C_G}{N^2}(C_G+\Vert Df_0\Vert_{\infty})^2(\frac{4}{N^2}\sum_{ij}\Vert D^2h_{ij}\Vert_{\infty}^2+\Vert D^2f_0\Vert_{\infty}^2)\\
    &\geq -\frac{3C_G}{N}E_Q^N(s,{\bf m_s})-\frac{2C_G}{N^2}(C_G+\Vert Df_0\Vert_{\infty})^2(\frac{4}{N^2}\sum_{ij}\Vert D^2h_{ij}\Vert_{\infty}^2+\Vert D^2f_0\Vert_{\infty}^2)\,.
    \end{align*}
We then apply Gronwall's inequality, we get for all $s \in [0,T]$ : 
\begin{align*}
    E_Q^N(s,{\bf m_s}) &\leq e^{3\frac{C_G}{N}(T-s)}E_Q^N(T,{\bf m_T})\\
    & \qquad+ \frac{N}{3C_G}\frac{2C_G}{N^2}(C_G+\Vert Df_0\Vert_{\infty})^2(\frac{4}{N^2}\sum_{ij}\Vert D^2h_{ij}\Vert_{\infty}^2+\Vert D^2f_0\Vert_{\infty}^2)\big ( e^{3\frac{C_G}{N}(T-s)}-1
    \big ) \nonumber \\
    & \leq e^{3\frac{C_G}{N}(T-s)}E_Q^N(T,{\bf m_T})\\
    & \qquad+ \frac{(C_G+\Vert Df_0\Vert_{\infty})^2}{N}(\frac{4}{N^2}\sum_{ij}\Vert D^2h_{ij}\Vert_{\infty}^2+\Vert D^2f_0\Vert_{\infty}^2)\big ( e^{3\frac{C_G}{N}(T-s)}-1
    \big )
    \,. 
\end{align*}

Now, recall that for all $s \in [t,T]$
\[
{\bf m_s} = ( \mathcal{L}(\check{X}^1_{s}), \dots, \mathcal{L}(\check{X}^N_{s}) )\,.
\]
The drift coefficient $\check a^{\mathcal{V}}$ is $C_G-$Lipschitz according to Lemma \ref{lem:Lip_aVcal} and is bounded (as shown before) so a fortiori of linear growth. Moreover, it satisfies the inequality in Lemma \ref{ineq:b_holder}. Therefore, it satisfies the conditions of Lemma $4.12$ of \cite{jackson2023approximatelyoptimaldistributedstochastic}. \\

Thus, if ${\bf m_t}$ satisfies the Poincaré inequality with some constant $c$, then,  for $t \leq s \leq T$,  
${\bf m_s}$ satisfies a Poincaré inequality with constant  
$$
\frac{e^{2C_G (s-t)} - 1}{2C_G} + c e^{2C_G (s-t)}\,.
$$

By assumption ${\bf m_t}=\mathcal{L}({\bf\check{X}_{t}})=\boldsymbol{\mu}$ satisfies the Poincaré inequality with some constant $c_p$. Therefore, $\mathcal{L}({\bf \check X_T)={\bf m_T}}$ satisfies the Poincaré inequality with constant $
\frac{e^{2C_G (T-t)} - 1}{2C_G} + c e^{2C_G (T-t)}\,,
$ so a fortiori with constant $C_p:=\frac{e^{2C_G T} - 1}{2C_G} + c_p e^{2C_G T}\,.$
 Thus, noticing that for all $i \in \{1,...,N\}$, $D_ig^N$ is a bounded Lipschitz function, we deduce that
\begin{align*}
E_Q^N(T,{\bf m_T}) &= N \sum_{i=1}^{N} \mathbb{E} \left[ |D_i g^N({\bf \check{X}_T})|^2 - \left| \mathbb{E} [D_i g^N({\bf \check{X}_T}) | \check{X}^i_{T}] \right|^2 \right]\\
&\leq NC_p \mathbb{E} \sum_{ij} |D_{ij} 
 g^N({\bf \check{X}_T})|^2\\
 & \leq NC_p \sum_{i, j } \Vert D_{ij} 
 g^N\Vert ^2_{\infty}\,.
\end{align*}
\end{proof}

\subsection{Proof of Theorem \ref{thm:diff_V_O}}
We are now ready to prove Theorem \ref{thm:diff_V_O}.
\begin{proof}
By Lemma \ref{lem:diffvalfx}, it suffices to analyse $$\int_t^T E^N(s, {\bf m_s})\,,$$ for all $t \in [0,t]\,.$ 
Recalling the definition of $E_1^N$ and $E_2^N$, respectively from \eqref{def:E1} and \eqref{def:E2}, we get
\begin{align*}
    \int_t^T E^N(s, {\bf m_s}) \, ds &=\int_t^T \big (E_1^N(s,{\bf m_s})+E_2^N(s,{\bf m_s})\big )ds \,,
\end{align*}
where
$$
\int_t^T E_1^N(s,{\bf m_s})ds \leq (T-t)K_1\,,
$$
and, by the Proposition \ref{prop:EQ_mT}, we know that, for all $s \in [0,T]\,,$
\begin{align*}
    E_Q^N(s,{\bf m_s}) &\leq e^{3\frac{C_G}{N}(T-s)}E_Q^N(T,{\bf m_T})\\
    & \qquad +\frac{(C_G+\Vert Df_0\Vert_{\infty})^2}{N}(\frac{4}{N^2}\sum_{ij}\Vert D^2h_{ij}\Vert_{\infty}^2+\Vert D^2f_0\Vert_{\infty}^2)\big ( e^{3\frac{C_G}{N}(T-s)}-1
    \big )\,.
\end{align*}

Therefore, we also have, for all $s \in [0,T]\,,$:
\begin{align*}
    \sqrt{E_Q^N(s,{\bf m_s})} &\leq e^{\frac{3}{2}\frac{C_G}{N}(T-s)}\sqrt{E_Q^N(T,{\bf m_T})}\\
    & \qquad + \sqrt{\frac{(C_G+\Vert Df_0\Vert_{\infty})^2}{N}(\frac{4}{N^2}\sum_{ij}\Vert D^2h_{ij}\Vert_{\infty}^2+\Vert D^2f_0\Vert_{\infty}^2)\big ( e^{3\frac{C_G}{N}(T-s)}-1
    \big )}
    \,. 
\end{align*}

Thus, we deduce that
\begin{align*}
    &\int_t^T E_2^N(s,{\bf m_s})ds \\
    &\leq \int_t^T\big (E_Q^N(s,{\bf m_s})+2(C_G +\Vert Df_0\Vert_\infty)\sqrt{E_Q^N(s,{\bf m_s})}\big)ds\\
    &\leq \int_t^Te^{\frac{3C_G}{2N}(T-s)}\sqrt{NC_P\sum_{ij}\Vert D_{ij}g^N\Vert_{\infty}^2}\biggl(2C_G+e^{\frac{3C_G}{2N}(T-s)}\sqrt{NC_P\sum_{ij}\Vert D_{ij}g^N\Vert_{\infty}^2}\biggr)ds\\
    &+\int_t^T\frac{(C_G +\Vert Df_0\Vert_\infty)^2}{\sqrt{N}}\sqrt{(\frac{4}{N^2}\sum_{ij}\Vert D^2h_{ij}\Vert_{\infty}^2+\Vert D^2f_0\Vert_{\infty}^2)\big ( e^{3\frac{C_G}{N}(T-s)}-1
    \big )}\\
    & \qquad \times\big (2+\sqrt{\frac{1}{N}(\frac{4}{N^2}\sum_{ij}\Vert D^2h_{ij}\Vert_{\infty}^2+\Vert D^2f_0\Vert_{\infty}^2)\big ( e^{3\frac{C_G}{N}(T-s)}-1
    \big )}\big )ds \\
     &\leq (T-t)e^{\frac{3C_G}{2N}(T-t)}\sqrt{NC_P\sum_{ij}\Vert D_{ij}g^N\Vert_{\infty}^2}\biggl(2C_G+e^{\frac{3C_G}{2N}(T-t)}\sqrt{NC_P\sum_{ij}\Vert D_{ij}g^N\Vert_{\infty}^2}\biggr)\\
     &+(T-t)\frac{(C_G +\Vert Df_0\Vert_\infty)^2}{\sqrt{N}}\sqrt{(\frac{4}{N^2}\sum_{ij}\Vert D^2h_{ij}\Vert_{\infty}^2+\Vert D^2f_0\Vert_{\infty}^2)\big ( e^{3\frac{C_G}{N}(T-t)}-1
    \big )}\\
    & \qquad \times\big (2+\sqrt{\frac{1}{N}(\frac{4}{N^2}\sum_{ij}\Vert D^2h_{ij}\Vert_{\infty}^2+\Vert D^2f_0\Vert_{\infty}^2)\big ( e^{3\frac{C_G}{N}(T-t)}-1
    \big )}\big )\,.
\end{align*}
Therefore, $$\int_t^T E^N(s,{\bf m_s})ds \leq (T-t)\big (K_f(t)+K_g(t)\big )\,,
$$
where $K_f$ and $K_G$ are defined in Theorem \ref{thm:diff_V_O}.
\end{proof}

\bibliographystyle{alpha}
\bibliography{bib_control}

\newcommand{\etalchar}[1]{$^{#1}$}
\begin{thebibliography}{DPTAS19}

\bibitem[ABVC19]{acciaio2019extended}
Beatrice Acciaio, Julio Backhoff-Veraguas, and Ren{\'e} Carmona.
\newblock Extended mean field control problems: stochastic maximum principle and transport perspective.
\newblock {\em SIAM journal on Control and Optimization}, 57(6):3666--3693, 2019.

\bibitem[ACL22]{aurell2022stochastic}
Alexander Aurell, Ren{\'e} Carmona, and Mathieu Lauriere.
\newblock Stochastic graphon games: Ii. the linear-quadratic case.
\newblock {\em Applied Mathematics \& Optimization}, 85(3):39, 2022.

\bibitem[AK21]{achdou2020mean}
Yves Achdou and Ziad Kobeissi.
\newblock Mean field games of controls: Finite difference approximations.
\newblock {\em Mathematics in Engineering}, 3(3):1--35, 2021.

\bibitem[BGP23]{doi:10.1137/21M1407720}
J.~Fr\'{e}d\'{e}ric Bonnans, Justina Gianatti, and Laurent Pfeiffer.
\newblock A lagrangian approach for aggregative mean field games of controls with mixed and final constraints.
\newblock {\em SIAM Journal on Control and Optimization}, 61(1):105--134, 2023.

\bibitem[BR25]{bertucci2025strategicgeometricgraphsmean}
Charles Bertucci and Matthias Rakotomalala.
\newblock Strategic geometric graphs through mean field games.
\newblock {\em SIAM Journal on Control and Optimization}, 63(4):2577--2604, 2025.

\bibitem[BWZ23]{bayraktar2023propagation}
Erhan Bayraktar, Ruoyu Wu, and Xin Zhang.
\newblock Propagation of chaos of forward--backward stochastic differential equations with graphon interactions.
\newblock {\em Applied Mathematics \& Optimization}, 88(1):25, 2023.

\bibitem[CCD22]{chassagneux2014probabilistic}
Jean-Fran{\c{c}}ois Chassagneux, Dan Crisan, and Fran{\c{c}}ois Delarue.
\newblock {\em A probabilistic approach to classical solutions of the master equation for large population equilibria}, volume 280.
\newblock American Mathematical Society, 2022.

\bibitem[CD{\etalchar{+}}18]{carmona2018probabilistic}
Ren{\'e} Carmona, Fran{\c{c}}ois Delarue, et~al.
\newblock {\em Probabilistic theory of mean field games with applications {I-II}}.
\newblock Springer, 2018.

\bibitem[CDJS23]{cardaliaguet2023algebraic}
Pierre Cardaliaguet, Samuel Daudin, Joe Jackson, and Panagiotis~E Souganidis.
\newblock An algebraic convergence rate for the optimal control of mckean--vlasov dynamics.
\newblock {\em SIAM Journal on Control and Optimization}, 61(6):3341--3369, 2023.

\bibitem[CDLL19]{cardaliaguet2019master}
Pierre Cardaliaguet, Fran{\c{c}}ois Delarue, Jean-Michel Lasry, and Pierre-Louis Lions.
\newblock {\em The master equation and the convergence problem in mean field games}.
\newblock Princeton University Press, 2019.

\bibitem[CH21]{caines2021graphon}
Peter~E Caines and Minyi Huang.
\newblock Graphon mean field games and their equations.
\newblock {\em SIAM Journal on Control and Optimization}, 59(6):4373--4399, 2021.

\bibitem[CL15]{carmona2015probabilistic}
René Carmona and Daniel Lacker.
\newblock A probabilistic weak formulation of mean field games and applications.
\newblock {\em The Annals of Applied Probability}, 25(3):1189--1231, 2015.

\bibitem[CL18]{cardaliaguet2018mean}
Pierre Cardaliaguet and Charles-Albert Lehalle.
\newblock Mean field game of controls and an application to trade crowding.
\newblock {\em Mathematics and Financial Economics}, 12(3):335--363, 2018.

\bibitem[CM23]{camilli2023quasi}
Fabio Camilli and Claudio Marchi.
\newblock On quasi-stationary mean field games of controls.
\newblock {\em Applied Mathematics \& Optimization}, 87(3):47, 2023.

\bibitem[DDJ24]{daudin2024optimal}
Samuel Daudin, Fran{\c{c}}ois Delarue, and Joe Jackson.
\newblock On the optimal rate for the convergence problem in mean field control.
\newblock {\em Journal of Functional Analysis}, 287(12):110660, 2024.

\bibitem[Dje22]{djete2022extendedmeanfieldcontrol}
Mao~Fabrice Djete.
\newblock Extended mean field control problem: a propagation of chaos result.
\newblock {\em Electronic Journal of Probability}, 27:1--53, 2022.

\bibitem[Dje23]{djete2023large}
Mao~Fabrice Djete.
\newblock Large population games with interactions through controls and common noise: convergence results and equivalence between open-loop and closed-loop controls.
\newblock {\em ESAIM: Control, Optimisation and Calculus of Variations}, 29:39, 2023.

\bibitem[DPT22]{djete2022mckean}
Mao~Fabrice Djete, Dylan Possama{\"\i}, and Xiaolu Tan.
\newblock Mckean--vlasov optimal control: limit theory and equivalence between different formulations.
\newblock {\em Mathematics of Operations Research}, 47(4):2891--2930, 2022.

\bibitem[DPTAS19]{de2019mean}
Antonio De~Paola, Vincenzo Trovato, David Angeli, and Goran Strbac.
\newblock A mean field game approach for distributed control of thermostatic loads acting in simultaneous energy-frequency response markets.
\newblock {\em IEEE Transactions on Smart Grid}, 10(6):5987--5999, 2019.

\bibitem[FS06]{fleming2006controlled}
Wendell~H Fleming and H~Mete Soner.
\newblock {\em Controlled Markov processes and viscosity solutions}.
\newblock Springer, 2006.

\bibitem[GMP21]{graber2021weak}
P~Jameson Graber, Alan Mullenix, and Laurent Pfeiffer.
\newblock Weak solutions for potential mean field games of controls.
\newblock {\em Nonlinear Differential Equations and Applications NoDEA}, 28(5):50, 2021.

\bibitem[GPV14]{gomes2014existence}
Diogo~A Gomes, Stefania Patrizi, and Vardan Voskanyan.
\newblock On the existence of classical solutions for stationary extended mean field games.
\newblock {\em Nonlinear Analysis: Theory, Methods \& Applications}, 99:49--79, 2014.

\bibitem[GPW22]{germain2022rate}
Maximilien Germain, Huy{\^e}n Pham, and Xavier Warin.
\newblock Rate of convergence for particle approximation of pdes in wasserstein space.
\newblock {\em Journal of Applied Probability}, 59(4):992--1008, 2022.

\bibitem[Gra16]{graber2016linear}
P~Jameson Graber.
\newblock Linear quadratic mean field type control and mean field games with common noise, with application to production of an exhaustible resource.
\newblock {\em Applied Mathematics \& Optimization}, 74(3):459--486, 2016.

\bibitem[GS23]{graber2023master}
P~Jameson Graber and Ronnie Sircar.
\newblock Master equation for cournot mean field games of control with absorption.
\newblock {\em Journal of Differential Equations}, 343:816--909, 2023.

\bibitem[GTC20]{gao2020linear}
Shuang Gao, Rinel~Foguen Tchuendom, and Peter~E Caines.
\newblock Linear quadratic graphon field games.
\newblock {\em arXiv preprint arXiv:2006.03964}, 2020.

\bibitem[GV16]{gomes2015extendeddeterministicmeanfieldgames}
Diogo~A Gomes and Vardan~K Voskanyan.
\newblock Extended deterministic mean-field games.
\newblock {\em SIAM Journal on Control and Optimization}, 54(2):1030--1055, 2016.

\bibitem[HMC06]{huang2006large}
Minyi Huang, Roland~P Malham{\'e}, and Peter~E Caines.
\newblock Large population stochastic dynamic games: closed-loop mckean-vlasov systems and the nash certainty equivalence principle.
\newblock 2006.

\bibitem[JL25]{jackson2023approximatelyoptimaldistributedstochastic}
Joe Jackson and Daniel Lacker.
\newblock Approximately optimal distributed stochastic controls beyond the mean field setting.
\newblock {\em The Annals of Applied Probability}, 35(1):251 -- 308, 2025.

\bibitem[Kob22a]{kobeissi2022mean}
Ziad Kobeissi.
\newblock Mean field games with monotonous interactions through the law of states and controls of the agents.
\newblock {\em Nonlinear Differential Equations and Applications NoDEA}, 29(5):52, 2022.

\bibitem[Kob22b]{kobeissi2022classical}
Ziad Kobeissi.
\newblock On classical solutions to the mean field game system of controls.
\newblock {\em Communications in Partial Differential Equations}, 47(3):453--488, 2022.

\bibitem[Lac17]{lacker2016limittheorycontrolledmckeanvlasov}
Daniel Lacker.
\newblock Limit theory for controlled mckean--vlasov dynamics.
\newblock {\em SIAM Journal on Control and Optimization}, 55(3):1641--1672, 2017.

\bibitem[LL07]{Lasry2007MeanFG}
J.~M. Lasry and Pierre-Louis Lions.
\newblock Mean field games.
\newblock {\em Japanese Journal of Mathematics}, 2:229--260, 2007.

\bibitem[LS22]{lacker2021soret}
Daniel Lacker and Agathe Soret.
\newblock A label-state formulation of stochastic graphon games and approximate equilibria on large networks.
\newblock {\em Mathematics of Operations Research}, 2022.

\bibitem[LSX19]{li2019linear}
Xun Li, Jingrui Sun, and Jie Xiong.
\newblock Linear quadratic optimal control problems for mean-field backward stochastic differential equations.
\newblock {\em Applied Mathematics \& Optimization}, 80(1):223--250, 2019.

\bibitem[Pha09]{pham2009continuous}
Huy{\^e}n Pham.
\newblock {\em Continuous-time stochastic control and optimization with financial applications}, volume~61.
\newblock Springer Science \& Business Media, 2009.

\bibitem[PW18]{pham2018bellman}
Huy{\^e}n Pham and Xiaoli Wei.
\newblock Bellman equation and viscosity solutions for mean-field stochastic control problem.
\newblock {\em ESAIM: Control, Optimisation and Calculus of Variations}, 24(1):437--461, 2018.

\bibitem[SAB{\etalchar{+}}23]{Seguret_2023}
Adrien Seguret, Clemence Alasseur, J.~Frédéric Bonnans, Antonio De~Paola, Nadia Oudjane, and Vincenzo Trovato.
\newblock Decomposition of convex high dimensional aggregative stochastic control problems.
\newblock {\em Applied Mathematics \&amp; Optimization}, 88(1), April 2023.

\bibitem[SS21]{santambrogio2021cucker}
Filippo Santambrogio and Woojoo Shim.
\newblock A cucker--smale inspired deterministic mean field game with velocity interactions.
\newblock {\em SIAM Journal on Control and Optimization}, 59(6):4155--4187, 2021.

\bibitem[Yon13]{yong2013linear}
Jiongmin Yong.
\newblock Linear-quadratic optimal control problems for mean-field stochastic differential equations.
\newblock {\em SIAM journal on Control and Optimization}, 51(4):2809--2838, 2013.

\bibitem[YZ99]{yong1999stochastic}
Jiongmin Yong and Xun~Yu Zhou.
\newblock {\em Stochastic controls: Hamiltonian systems and HJB equations}, volume~43.
\newblock Springer Science \& Business Media, 1999.

\end{thebibliography}

\end{document}